\documentclass[12pt]{amsart}
\usepackage{amssymb, amscd, latexsym, graphicx, psfrag}
\usepackage[all]{xy}

\newtheorem{dummy}{dummy}[section]
\newtheorem{lemma}[dummy]{Lemma}
\newtheorem{theorem}[dummy]{Theorem}

\newtheorem{corollary}[dummy]{Corollary}
\newtheorem{proposition}[dummy]{Proposition}
\theoremstyle{definition}
\newtheorem{definition}[dummy]{Definition}
\newtheorem{example}[dummy]{Example}
\newtheorem{remark}[dummy]{Remark}

\newcommand{\bA}{\mathbb{A}}
\newcommand{\bC}{\mathbb{C}}

\newcommand{\bP}{\mathbb{P}}
\newcommand{\bQ}{\mathbb{Q}}
\newcommand{\bR}{\mathbb{R}}
\newcommand{\bZ}{\mathbb{Z}}

\newcommand{\bfR}{\mathbf{R}}

\newcommand{\cE}{\mathcal{E}}
\newcommand{\cF}{\mathcal{F}}

\newcommand{\cH}{\mathcal{H}}

\newcommand{\cL}{\mathcal{L}}
\newcommand{\cM}{\mathcal{M}}
\newcommand{\cP}{\mathcal{P}}
\newcommand{\cQ}{\mathcal{Q}}
\newcommand{\cO}{\mathcal{O}}

\newcommand{\cT}{\mathcal{T}}

\newcommand{\sfR}{\mathbf{k}}

\newcommand{\Hom}{\mathrm{Hom}}

\newcommand{\can}{{\mathrm{can}} }

\newcommand{\bGamma}{\mathbf{\Gamma}}

\newcommand{\Ga}{\Gamma}

\newcommand{\LS}{ {\Lambda_\Sigma} }

\newcommand{\Spec}{\mathrm{Spec}\,}
\newcommand{\Ext}{\mathrm{Ext}}

\newcommand{\Sh}{\mathit{Sh}}

\newcommand{\Perf}{\cP\mathrm{erf}}

\newcommand{\FT}{\cF\cT}
\renewcommand{\SS}{\mathit{SS}}
\newcommand{\ltr}{\langle \Theta \rangle}

\newcommand{\ltrp}{\langle\Theta'\rangle}

\newcommand{\Gm}{\mathbb{G}_{\mathrm{m}}}
\newcommand{\vC}{\check{\mathrm{C}}}

\newcommand{\RGa}{\mathbf{R}\Gamma}

\newcommand{\Rmod}{\mathsf{R}\text{-}\mathrm{mod}}

\newcommand{\codim}{\mathrm{codim}}

\newcommand{\Mo}{\mathrm{M}}
\newcommand{\Moi}{\mathit{mo}}
\newcommand{\PolF}{\mathrm{PolF}}

\numberwithin{equation}{section}

\begin{document}

\title{Morse theory and toric vector bundles}
\author{David Treumann}

\begin{abstract}
Morelli's computation of the $K$-theory of a toric variety $X$ associates a polyhedrally constructible function on a real vector space to every equivariant vector bundle $\cE$ on $X$.  The coherent-constructible correspondence lifts Morelli's constructible function to a complex of constructible sheaves $\kappa(\cE)$.  We show that certain filtrations of the cohomology of $\kappa(\cE)$ coming from Morse theory coincide with the Klyachko filtrations of the generic stalk of $\cE$.  We give Morse-theoretic (i.e. microlocal) conditions for a complex of constructible sheaves to correspond to a vector bundle, and to a nef vector bundle.
\end{abstract}

\maketitle


\section{Introduction}

Let $X$ be an $n$-dimensional toric variety, and write $T$ for the algebraic torus that acts on $X$.  In \cite{Mo}, Morelli constructed an injective homomorphism from the Grothendieck group $K_T(X)$ of equivariant vector bundles on $X$ to the group of polyhedrally-constructible functions on an $n$-dimensional real vector space $M_\bR$.  Morelli's homomorphism provided an interesting computation of the $K$-theory of $X$, and also suggested a framework for studying the customary relationship between toric and polyhedral geometry: the algebro-geometric aspects of $X$ that are visible in its $K$-theory will have some expression in terms of constructible functions, and a lot of polyhedral geometry can be profitably encoded in terms of such functions \cite{Mo2, KP, McM}.

There is a rich theory of constructible functions \cite{GM, BDK} and a close analogy between such functions and constructible sheaves.  The most powerful tool in this theory is a kind of Morse theory, which makes certain features of constructible functions and sheaves visible ``microlocally'' i.e. in the cotangent bundle.  The purpose of this paper is to explore the relevance of these tools to Morelli's map and polyhedra, a connection which seems to have been unnoticed before our paper \cite{fltz}. 

We showed in \cite{fltz} that Morelli's isomorphism lifts to a categorical equivalence, called the ``coherent-constructible correspondence'' or CCC.  The equivalence (which we review here in Section \ref{sec:CCCrev}) matches bounded complexes of vector bundles $\cE^\bullet$ on $X$ with bounded cochain complexes of constructible sheaves $\kappa(\cE^\bullet)$ on $M_\bR$.  In this paper, we investigate which complexes of constructible sheaves come from genuine vector bundles, regarded as complexes concentrated in a single degree.  The answer is expressed nicely in terms of Morse theory, and interacts well with a classification of equivariant vector bundles given by Klyachko.

\subsection{Vector bundles and convex sheaves}

We introduce here a class of sheaves real vector spaces that we call \emph{convex}.  Our first main result is Theorem \ref{thm:1.1} below, which states that a sheaf on $M_\bR$ is matched to a vector bundle on $X$ under the CCC if and only if the sheaf is convex.  Recall that each complex of sheaves $F^\bullet$ on $M_\bR$ leads to a family of functors $U \mapsto H^i_c(U;F^\bullet)$, the ``compactly supported cohomology of $U$ with coefficients in $F^\bullet$'', and that these functors are covariant for open inclusions.  

\begin{definition}
Let $F^\bullet$ be a complex of sheaves on a real vector space $M_\bR$.  We say that $F^\bullet$ is \emph{convex} if the following hold:
\begin{enumerate}
\item For any linear function $f:M_\bR \to \bR$ and any $k \in \bR$, the natural map
$$H^0_c(f^{-1}(-\infty,k);F^\bullet) \to H^0_c(M_\bR;F^\bullet)$$
is injective
\item For $f$ and $k$ as in (1) the groups $H^i(M_\bR;F^\bullet)$ and $H^i(f^{-1}(-\infty,k);F^\bullet)$ vanish when $i \neq 0$.
\end{enumerate}
\end{definition}

\begin{example}
Let $F$ be the constant sheaf on an open subset $U \subset M_\bR$ extended by zero to all of $M_\bR$, and placed in cohomological degree $n = -\dim(M_\bR)$.  Then $H^i_c(M_\bR;F) \cong H^{i+n}_c(U;\bC)$ and $H^i_c(f^{-1}(-\infty,k);F) \cong H^{i+n}_c(f^{-1}(-\infty,k) \cap U;\bC)$.  
\begin{enumerate}
\item If $U$ is convex then $H^i_c(U;\bC) = 0$ for $i \neq n$, whereas $f^{-1}(-\infty,k) \cap Y$ is either empty or another convex open subset of dimension $n$ so that $H^i_c(U;\bC) = 0$ for $i \neq n$ and we have an injection $H^n_c(U;\bC) \to H^n_c(U;\bC)$.  It follows that $F$ is convex. 

\item If $U = \{(x,y) \mid y > 0 \text{ and } 1 < x^2 + y^2 < 2\}$ is an open half-annulus in $M_\bR = \bR^2$, then $F$ can be seen to be nonconvex by considering the linear function $(x,y) \mapsto y$

\end{enumerate}

Note that if $U$ is the disjoint union of two convex open subsets then $F$ is convex, so that convexity of $U$ is not necessary for convexity of $F$.  More complicated examples of convex sheaves appear in Section \ref{subsec:introexamples}.
\end{example}

\begin{theorem}
\label{thm:1.1}
Let $X$ be a toric variety, let $\cE^\bullet$ be a bounded complex of equivariant vector bundles on $X$, and let $\kappa(\cE^\bullet)$ be the corresponding constructible complex on $M_\bR$.  Then $\kappa(\cE^\bullet)$ is convex if and only if $\cE^\bullet$ is quasi-isomorphic to a vector bundle concentrated in degree zero.
\end{theorem}

\begin{remark}
\label{rem:dependence}
Note that the criterion on $\kappa(\cE^\bullet)$ for $\cE^\bullet$ to be a vector bundle does not depend on the toric variety on which $\cE^\bullet$ is defined.  That is, there is not a notion of ``$X$-convex sheaves on $M_\bR$'' that changes with toric variety $X$.  Part of the reason for this is that (derived) pullback along a birational map $f$ preserves the property of being a vector bundle, and it is proved in \cite{fltz} that $\kappa(\cE^\bullet) \cong \kappa(f^* \cE)$ for all $\cE^\bullet$.  

It would be interesting to have a criterion in terms of $\kappa(\cE^\bullet)$ for $\cE^\bullet$ to be quasi-isomorphic to a coherent sheaf concentrated in degree zero.  When $X$ is smooth, this would be a description of the heart of a $t$-structure on a subcategory of the derived category sheaves on $M_\bR$.  But since being a coherent sheaf is not preserved by derived pullbacks, such a criterion would necessarily depend on $X$.
\end{remark}

Each linear function $M_\bR \to \bR$ equips the global sections of a convex sheaf with an $\bR$-indexed filtration.  We review some of the theory of these ``Morse filtrations'' in Section \ref{sec:morsetheta}.  On the coherent side of the CCC the corresponding filtrations are well-known in toric geometry.  Recall that to a toric variety we associate a fan $\Sigma$ in the dual vector space $N_\bR$ to $M_\bR$.  In \cite{K}, Klyachko showed that to give a toric vector bundle $\cE$ on $X$ is equivalent to giving a vector space $E$ equipped with a compatible family of $\bZ$-indexed filtrations $\{E^\alpha_{\leq k}\}$ where $\alpha$ runs through the rays of $\Sigma$.  

\begin{theorem}
\label{thm:1.1.5}
Let $\cE$ be an equivariant vector bundle on a toric variety $X$ and let $\kappa(\cE)$ be the corresponding constructible complex.  Let $E$ be the stalk of $\cE$ at a generic point.  Then there is a natural isomorphism $H^0(M_\bR;F^\bullet) \cong E$.  If $\alpha \in N_\bR$ generates a ray of $\Sigma$ and $k$ is an integer, then under this natural isomorphism the inclusion
$$H^0_c(\alpha^{-1}(-\infty,k);\kappa(\cE)) \to H_c^0(M_\bR;\kappa(\cE))$$
is compatible with the inclusion of $E^\alpha_{\leq k}$ into $E$.
\end{theorem}

\subsection{Microlocal stalks}

The microlocal theory of sheaves \cite{KS} associates to each constructible complex of sheaves $F^\bullet$ and each point $x$ of $M_\bR$ a complex of sheaves $\mu_x F^\bullet$ on the contangent space to $M_\bR$ at $x$, which we may identify with $N_\bR$.  This sheaf $\mu_x F^\bullet$ is called the ``microlocalization'' of $F^\bullet$ at $x$, and its stalks are called microlocal stalks.  If $f$ is a smooth function on $M_\bR$ then the microlocal stalks $\mu_{x,df_x} F^\bullet$ are closely related to the Morse theory of $f$ and $F^\bullet$.  The sheaves $F^\bullet$ appearing in the CCC have the property (indeed they are almost characterized by the property) that $\mu_x F^\bullet$ is constructible with respect to the cones in the fan $\Sigma$ associated to $X$.  Such a sheaf can be described combinatorially as
\begin{itemize}
\item A vector space $G_\sigma$ for each cone $\sigma \in \Sigma$.  $G_\sigma$ is the stalk of the sheaf at a point on the interior of $\sigma$.
\item A map $G_\sigma \to G_\tau$ for every pair of cones with $\sigma \subset \tau$, such that all triangles associated to triples of cones $\sigma \subset \tau \subset \upsilon$ commute.
\end{itemize}
A similar description can be obtained for any sheaf constant along the cells of a regular cell complex.  A constructible complex of sheaves can be represented by a chain complex of such data.  Our next result computes the sheaf $\mu_x(\kappa(\cE^\bullet))$ in terms of $\cE^\bullet$.

\begin{theorem}
\label{thm:1.2}
Let $X$ be a toric variety with fan $\Sigma$, let $\cE^\bullet$ be a bounded complex of equivariant vector bundles on $X$ and let $\kappa(\cE^\bullet)$ be the corresponding constructible complex.  Let $\sigma$ be a cone of $\Sigma$, and let $X_\sigma$ denote the closure of the $T$-orbit corresponding to $\sigma$.  Let $x$ be a character of $T$, regarded as a lattice point in $M_\bR$, and let $\xi \in N_\bR$ belong to the interior of $-\sigma$.  The following hold:
\begin{enumerate}
\item We have natural isomorphisms
$$\bfR\Gamma(X_\sigma,\cE^\bullet\vert_{X_\sigma})_x \cong \mu_{x,\xi}(\kappa(\cE^\bullet))$$
where on the left-hand side $(-)_x$ denote the $x$th weight space of the $T$-module $(-)$.
\item If $-\psi$ is a point in the interior of a cone $\tau \supset \sigma$, then under these natural isomorphisms the restriction map $\mu_{x,\xi}(\kappa(\cE^\bullet)) \to \mu_{x,\psi}(\kappa(\cE^\bullet))$ coincides with the restriction map $\bfR\Gamma(X_\sigma,\cE^\bullet\vert_{X_\sigma}) \to \bfR\Gamma(X_\tau,\cE^\bullet_{X_\tau})$
\end{enumerate}
\end{theorem}

\begin{remark}
The microlocal stalk $\mu_{x,0}F^\bullet$ at a zero covector is naturally isomorphic to the ``costalk'' of $F^\bullet$ at $x$.  This is the right-derived functor of the assignment that carries a sheaf to the group of sections that are supported at the single point $x$, and differs from the usual stalk functor by Verdier duality.  Thus a special case of Theorem \ref{thm:1.2} is the formula
$$\bfR\Gamma(X;\cE)_x \cong \mathit{costalk}_x \kappa(\cE^\bullet)$$
The constructible sheaf $\kappa(\cE^\bullet)$ is defined in such a way to make this formula reminiscent of formulas of Demazure (in the case when $\cE$ is a line bundle, \cite{D}) and Klyachko (\cite[Section 4.3]{K}).
\end{remark}

\begin{remark}
Theorem \ref{thm:1.2} together with Grothendieck vanishing imply that constructible sheaves of the form $F^\bullet = \kappa(\cE)$ have the property that $\mu_x(F^\bullet)$ is concentrated in degrees between $0$ and $n$.  It would be interesting to know if this is true for general convex sheaves $F^\bullet$.
\end{remark}

\subsection{Microlocal characterizations of vector bundles and nef vector bundles}

A corollary of Theorem \ref{thm:1.2} is a second (after Theorem \ref{thm:1.1}) characterization of those sheaves that come from equivariant vector bundles, this time in terms of microlocal stalks.  Combining Theorem \ref{thm:1.2} with a result of Hering, Mustata and Payne allows us to give a similar characterization of those sheaves that come from numerically effective vector bundles.

\begin{theorem}
\label{thm:vecbund}
Let $X$ be a proper toric variety with fan $\Sigma$, let $\cE^\bullet$ be a bounded complex of equivariant vector bundles on $X$, and let $\kappa(\cE^\bullet)$ be the corresponding constructible complex.  The following are equivalent:
\begin{enumerate}
\item $\cE^\bullet$ is quasi-isomorphic to a vector bundle concentrated in degree zero
\item For each $\xi \in N_\bR$ that belongs to the interior of a top-dimensional cone of $\Sigma$, and for each lattice point $x \in M_\bZ$, the following holds: the microlocal stalk $\mu_{x,-\xi}\kappa(\cE^\bullet)$ is concentrated in degree zero.
\end{enumerate}
\end{theorem}

Let us recall the notion of a \emph{numerically effective} or nef vector bundle.  A line bundle on a proper algebraic variety is nef if it has nonnegative degree on every embedded curve. A vector bundle $\cE$ on a proper algebraic variety is called nef if the line bundle $\cO_{\bP(\cE)}(1)$ on the projectivization of $\cE$ is nef.  On $\bP^1$, recall that every vector bundle splits as a sum of line bundles $\cO(n_1) \oplus \cdots \oplus \cO(n_k)$---such a vector bundle is nef if and only if each $n_i \geq 0$.

In \cite{HMP}, it was proved that an equivariant vector bundle on a proper toric variety is nef if and only if its restriction to every $T$-invariant curve (necessarily isomorphic to $\bP^1$) is nef.  Using this and Theorem \ref{thm:1.2}, we can establish the following:

\begin{theorem}
\label{thm:nefness}
Let $X$ be a proper toric variety, let $\cE^\bullet$ be a bounded complex of equivariant vector bundles on $X$, and let $\kappa(\cE^\bullet)$ be the corresponding constructible complex.  The following are equivalent:
\begin{enumerate}
\item $\cE^\bullet$ is quasi-isomorphic to a nef vector bundle concentrated in degree zero.
\item For each lattice point $x \in M_\bZ$, the following hold:
\begin{enumerate}
\item the microlocal stalks $\mu_{x,-\xi}\kappa(\cE^\bullet)$ are concentrated in degree zero whenever $\xi$ belongs to the interior of a codimension zero or codimension one cone.
\item if $\xi$ belongs to a codimension one cone and $\psi$ belongs to an incident codimension zero cone, then the restriction map $\mu_{x,-\xi} \kappa(\cE^\bullet) \to \mu_{x,-\psi}\kappa(\cE^\bullet)$ is surjective.
\end{enumerate}
\end{enumerate}
\end{theorem}

\begin{remark}
\label{rem:2a2b}
Conditions (2a) and (2b) are easy to check in practice.  For instance, the constant sheaf on a closed union of cones in $N_\bR$ satisfies them, as does any sheaf supported on a union of cones of codimension at least 2.
\end{remark}

\begin{remark}
\label{rem:local}
Unlike Theorem \ref{thm:1.1}, the criteria of Theorems \ref{thm:vecbund}(2) and \ref{thm:nefness}(2) are manifestly local.  In fact, they can be checked in a neighbhorhood of each lattice point.
\end{remark}

\subsection{Examples}
\label{subsec:introexamples}

In dimension two it's sometimes possible to give ad hoc descriptions of constructible sheaves in pictures.  We give here some examples of convex sheaves on $\bR^2$ coming from equivariant vector bundles (in fact nef equivariant vector bundles) on toric surfaces.

\begin{example}
\label{ex:fujino}
The figure shows the constructible sheaf associated to a vector bundle on $\bP^2$ considered by Fujino (\cite[Example 4.10]{HMP}).
\begin{center}
\includegraphics[scale=.5]{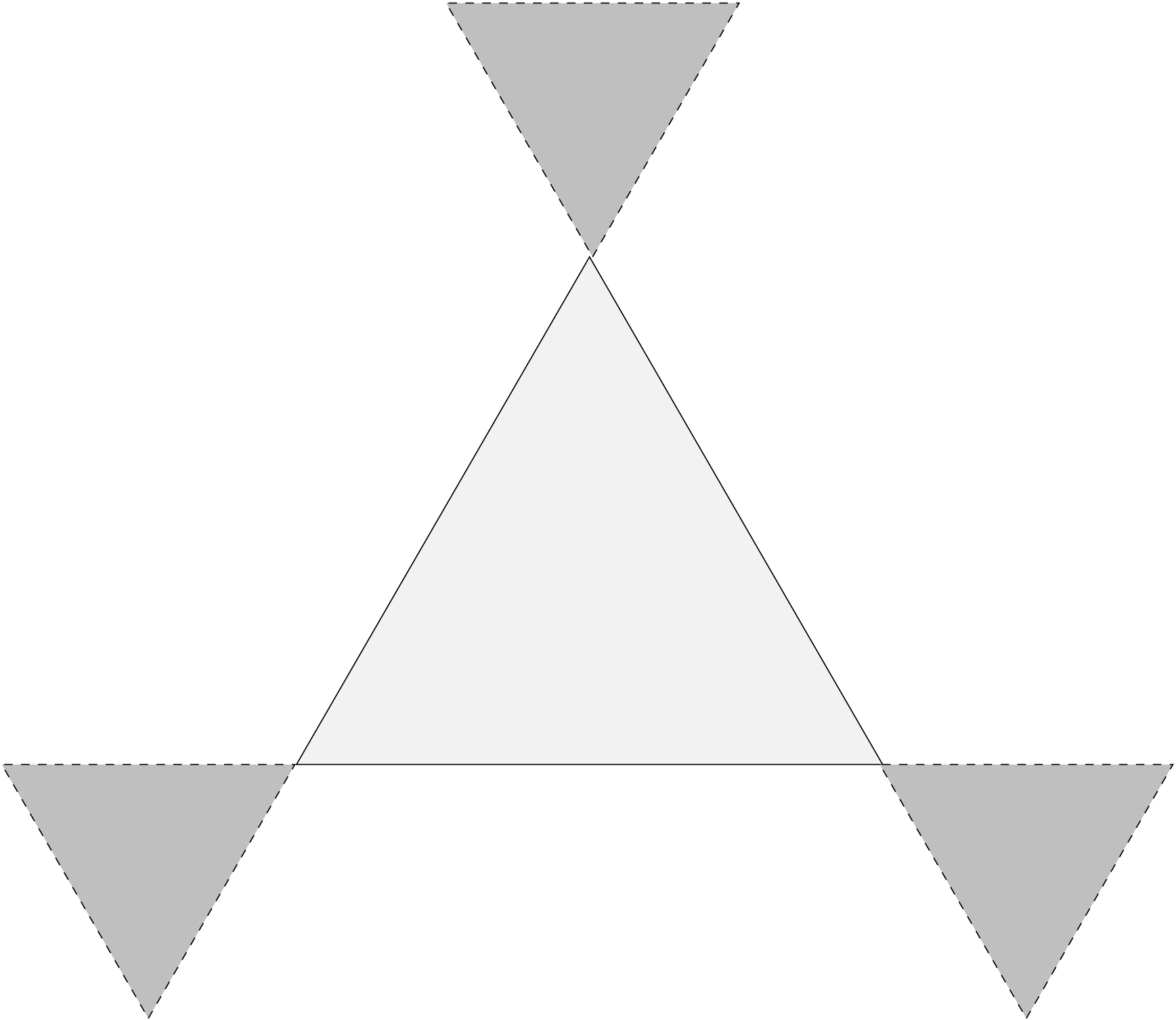}
\end{center}
The vector bundle is of the form $\mathit{Fr}^* T_{\bP^2} \otimes \cO(-m) = (\mathit{Fr}^* T_{\bP^2})$ where $T_\bP^2$ is the tangent sheaf to $\bP^2$, $\mathit{Fr}:\bP^2 \to \bP^2$ is a toric Frobenius map that raises each coordinate to the $n$th power for some integer $n$, and $m \gg 0$.  The equivariant structure on $\mathit{Fr}^* T_{\bP^2}$ is induced by that on $T_\bP^2$, and there is a $\bZ^2$-torsor of equivariant structures on $\cO(-m)$ that we can chose from.  We list some of the features of this sheaf:
\begin{enumerate}
\item As we have not specified an equivariant structure on $\cO(-m)$, the associated constructible sheaf is only well-defined up to translation.  Changing the equivariant structure by a character of the torus translates the constructible sheaf the corresponding lattice vector.
\item The darkly-shaded triangles are all congruent and have side lengths that grow with $n$, and the lightly shaded triangle has side lengths that grow with $m$.  Each side of a dark triangle passes through $n+1$ lattice points including the vertices, and each side of a light triangle passes through $m+1$ lattice points including the vertices.
\item The darkly-shaded (resp. lightly-shaded) region in the sheaf indicates a rank-one sheaf placed in degree -2 (resp. degree -1), and the stalks of the sheaf vanish along the dotted lines.  
\item The behavior of the sheaf in a neighborhood of the three points where the dark triangles meet the light triangle is described in more detail in Example \ref{ex:weirdlocal}.  In particular, the microlocalization at these points is the constant sheaf on a union of two rays in $N_\bR$, extended by zero.
\end{enumerate}
The vector bundle is nef, and this is visible in the diagram using the criterion of Theorem \ref{thm:nefness}: the microlocalizations $\mu_x$ can be computed as in Section \ref{sec:seomx}, and in each case Remark \ref{rem:2a2b} applies.  Fujino pointed out that this is a nef vector bundle on $\bP^2$ whose higher cohomology does not vanish, in contrast to the case of nef toric line bundles whose higher cohomology is always trivial.  This is also visible in the diagram: any lattice point in the interior of the middle triangle (which exist for $m$ large enough) will contribute to $H^1(\bP^2,\cE)$ by Theorem \ref{thm:1.2}.
\end{example}

\begin{example}
\label{ex:ml}
Let $L$ be an ample equivariant line bundle on $X$---then $L$ is globally generated.  The kernel of the surjection
$$H^0(X;L) \otimes \cO_X \to L$$
is another equivariant vector bundle called $\cM_L$, studied in the toric case in \cite{HMP}.  It is slightly easier to diagram $\kappa(\cM_L \otimes L')$ where $L'$ is another ample line bundle---here it is in case $X = \bP^1 \times \bP^1$, $L = \cO(3,2)$, and $L' = \cO(1,1)$:
\begin{center}
\includegraphics[scale=.5,angle=90]{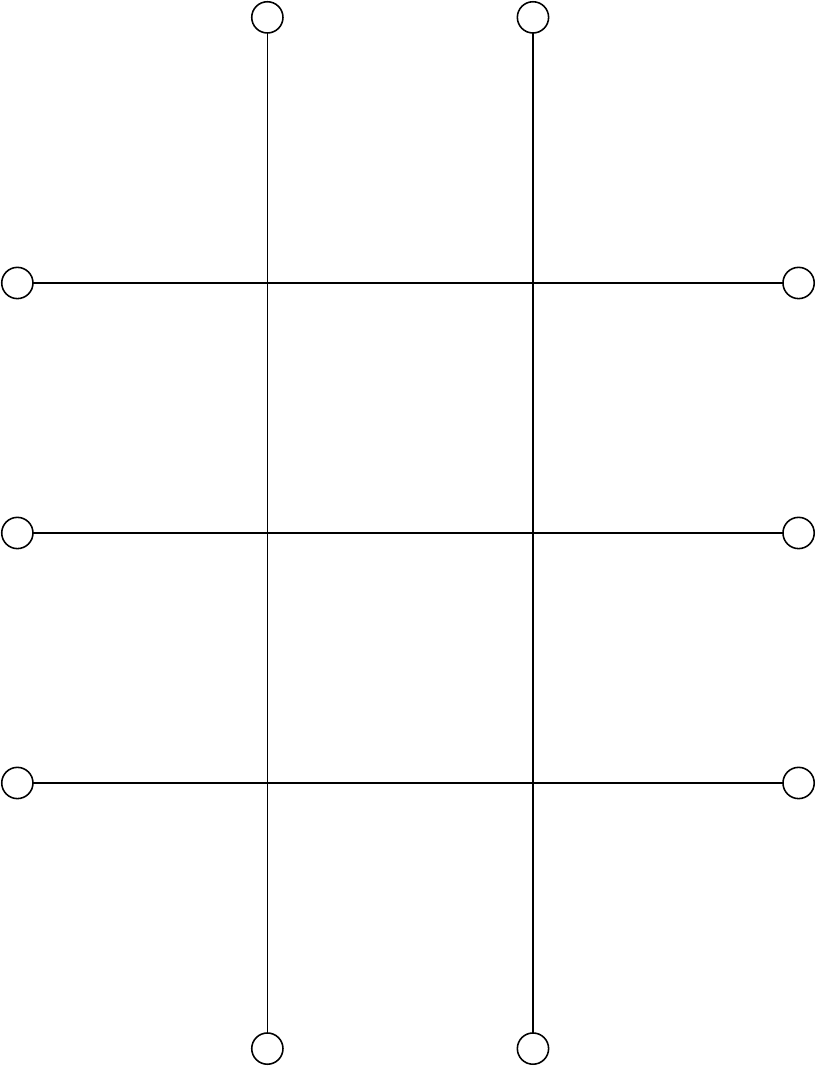}
\end{center}
That is, $\kappa(\cM_{\cO(3,2)} \otimes \cO(1,1))$ is a rank one constant sheaf on the union of three vertical open intervals and two horizontal open intervals, placed in degree $-2$, and extended by zero.  Note that the fact that this sheaf is convex depends crucially on the fact that the stalks at each endpoint of these intervals in zero.  

The vector bundle is nef and this is again visible in the figure by computing the microlocalizations at the noncrossing points of the intervals, at the endpoints of the intervals, and at the crossing points:
\begin{enumerate}
\item At the noncrossing points, the microlocalization is the constant sheaf supported on the conormal line to that point.

\item At the endpoints, the microlocalization is the constant sheaf supported on a closed half-plane

\item At the crossing points, the microlocalization is concentrated in degree zero, and is isomorphic to the kernel of the natural map
$$F_I \oplus F_{II} \oplus F_{III} \oplus F_{IV} \to \delta$$
where each $F$ is a constant sheaf supported on a closed quadrant and $\delta$ is the skyscraper sheaf at the origin.
\end{enumerate}

\end{example}

\begin{remark}
By Remark \ref{rem:local}, any sheaf locally isomorphic to the one displayed in Example \ref{ex:ml}---or locally isomorphic to a direct sum of $r$ copies of it---also comes from a nef vector bundle.  There is therefore a full embedding of the abelian category of representations of the free group on two generators (the fundamental group of the figure) to the additive category of nef vector bundles on $\bP^1 \times \bP^1$.  In the language of representation theory \cite{Dr}, this means that nef vector bundles on toric surfaces have ``wild representation type''---the problem  of classifying the indecomposable objects is as hard as classifying the indecomposable representations of a free group.
\end{remark}


\subsection{Notation and conventions}
\label{sec:filt}

We work over a commutative noetherian base ring $\sfR$.

Our filtrations will usually be indexed by the real numbers.  An \emph{open increasing filtration} of an $\sfR$-module $\cM$ is a sequence of submodules $\cM_{< t} \subset \cM$ such that whenever $s \leq t$, $\cM_{< s} \subset \cM_{< t}$, and such that $\cM_{< t} = \bigcup_{s < t} \cM_s$ for all $t$.  A \emph{closed increasing filtration} of $\cM$ is a sequence of submodules $\cM_{\leq t} \subset \cM$ such that whenever $s \leq t$, $\cM_{\leq s} \subset \cM_{\leq t}$ and such that $\cM_{\leq t} = \bigcap_{s > t} \cM_{\leq s}$.  Open increasing filtrations are equivalent to closed increasing filtrations: to go from one to the other set $\cM_{\leq t} = \bigcap_{s > t} \cM_{< s}$ and $\cM_{< t} = \bigcup_{s \leq t} \cM_{\leq s}$.

A toric variety $X$ is a variety equipped with an action of an algebraic torus $T \cong \Gm^n$.  We let $M$ be the character lattice of $T$ and $N$ the dual cocharacter lattice.  We write $M_\bR$ and $N_\bR$ for the realifications of $M$ and $N$.  To $X$ we can associate a fan $\Sigma \subset N_\bR$.  For each $\sigma \in \Sigma$, we have
\begin{itemize}
\item The associated $T$-orbit $O_\sigma \subset X_\sigma$, with $\codim(O_\sigma) = \dim(\sigma)$
\item The smallest $T$-stable Zariski open set $U_\sigma$ containing $O_\sigma$
\item The smallest $T$-stable Zariski closed set $X_\sigma$ containing $O_\sigma$
\end{itemize}
If $\chi:T \to \Gm$ is a character of $T$, we let $\cO(\chi)$ denote the structure sheaf of $X$ with an equivariant structure given by $\chi$ is the following way: if $t \in T$ and $(x,y) \in X \times \mathbb{A}^1$ then 
$$t \cdot (x,y) = (tx,\chi(t)^{-1} \cdot y)$$
If $X = \Spec \sfR[\sigma^\vee \cap M]$ is affine, then $\cO(\chi)$, regarded as an $M$-graded module, has a nonzero element in degree $\psi \in M$ if and only if $\psi - \chi \in \sigma^\vee$.

This paper requires less from modern homological algebra than \cite{fltz}---we work with derived categories in the sense of Verdier rather than any kind of triangulated dg category.  Some of our notation therefore clashes with that of \cite{fltz} but this should not cause confusion.  We let $\cQ_T(X)$ denote the abelian category of $T$-equivariant quasi-coherent sheaves on $X$, and $\Sh(V)$ for the abelian category of sheaves of $\sfR$-modules on a real vector space $V$.  The bounded derived categories are denoted by $D^b(\cQ_T(X))$ and $D^b(\Sh(V))$.  
\begin{itemize}
\item An object of $D^b(\Sh(V))$ is called \emph{constructible} if its cohomology sheaves are constant along the strata of  a Whitney stratification (for instance, the open simplices of a piecewise-smooth triangulation) of $V$.  We write $D^b_c(V)$ for the full subcategory of $D^b(\Sh(V))$ whose objects are constructible complexes.
\item A complex of quasicoherent sheaves on $X$ is \emph{perfect} if it is locally on $X$ quasi-isomorphic to a bounded complex of vector bundles.  We write $\Perf_T(X)$ for the full subcategory of $D^b(\cQ_T(X))$ spanned by equivariant quasicoherent sheaves whose underlying nonequivariant sheaf is perfect.  It's not known if every perfect complex on a toric variety is quasi-isomorphic to a bounded complex of vector bundles but this is true on smooth varieties (where furthermore every bounded complex of coherent sheaves is perfect) and on projective varieties.
\end{itemize}

If $Y$ is a locally compact topological space, let $H^i_c(Y;\bQ)$ denote the compactly-supported cohomology of $Y$ with rational coefficients.  If the latter groups are finite-dimensional and vanish in large degrees, we set
$$\chi_c(Y) = \sum (-1)^i \dim(H^i_c(Y;\bQ))$$

\section{Review of Morelli's isomorphism}

Let $X$ be a projective toric variety.  An equivariant ample line bundle $\cL$ on $X$ a toric variety determines a polytope $P_\cL$ in $M_\bR$ whose vertices are at lattice points.  We may encode a polytope $P$ by the normalized indicator function
$$j_P(x) := \bigg\{
\begin{array}{ll} 
(-1)^{\dim(M_\bR)} & \text{if $x$ is in the interior of $P$} \\
0 & \text{if $x$ is on the boundary or outside of $P$}
\end{array}
$$
Let us call a $\bZ$-valued function on $M_\bR$ \emph{polyhedral} (what Morelli calls \emph{hedral} in \cite{Mo2}) if it is piecewise-constant along a polyhedral stratification of $M_\bR$---we will define this in more detail in \ref{sec:ecm} below.  Write $\PolF(M_\bR)$ for the group of polyhedral functions on $M_\bR$.  Morelli's theorem is the following:

\begin{theorem}[Morelli \cite{Mo}]
\label{thm:morelli1}
Let $X$ be a projective toric variety with torus $T$.  Let $K_T(X)$ denote the Grothendieck group of $T$-equivariant vector bundles on $X$.  There is a unique injective homomorphism of groups $\Moi:K_T(X) \hookrightarrow \PolF(M_\bR)$ that carries the class of an ample equivariant line bundle $\cL$ to the polyhedral function $j_{P_\cL}$.
\end{theorem}

Much more is true.  For instance, there is a local criterion for an element of $\PolF(M_\bR)$ to be in the image of the map from $K_T(X)$, and the map is compatible with ring structures on $K_T(X)$ and $\PolF(M_\bR)$.  These refined results are best expressed using the concept Euler characteristic measure.  In this section we will review the language of Euler characteristic measure, summarize some of Morelli's results, and give a construction of Morelli's map $\Moi$.  We mostly follow Morelli \cite{Mo,Mo2} but make some changes that reflect our sheaf- and Morse-theoretic goals.  We omit many details, for a more complete story we refer of course to \cite{Mo,Mo2} and also to \cite{GM,KS} and their references.

\subsection{Euler characteristic measure and operations on constructible functions}
\label{sec:ecm}

The compactly-supported Euler characteristic of topological spaces can be thought of as a kind of measure, in the sense of measure theory; this point of view was advocated early on in \cite{Ha1,Ha2}.  The Euler characteristic can take negative values, cannot take noninteger values, and is not countably additive, but it does satisfy the relation
$$\chi_c(A \cup B) = \chi_c(A) + \chi_c(B) - \chi_c(A \cap B)$$
whenever $A$ and $B$ are sufficiently nice subsets of a nice topological space $X$.  This turns out to be enough to develop interesting analogs of familiar notions from function theory---notably, the notions of integration and of Fourier transform.

To make sense of the phrase ``sufficiently nice'' requires some stratification theory and, for some purposes, model theory.  For our purposes, however, $X$ can always be taken to be a real vector space, and ``sufficiently nice'' to mean polyhedral, in which case the theory is much simpler.  Let $V$ be a finite-dimensional real vector space.  A \emph{polyhedral cell} in $V$ is a subset of $V$ that is cut out by finitely many linear equations and finitely many strict linear inequalities.  Note that a polyhedral cell is allowed to be unbounded in $V$.  A \emph{polyhedral stratification} $S$ of $V$ is a decomposition of $V$ into finitely many polyhedral cells that are disjoint, and that satisfy the following property: the closure of a polyhedral cell in $S$ is a union of polyhedral cells in $S$.  (Such a stratification is always a Whitney stratification.)  A function on $V$ is polyhedral if it is constant along the cells of a polyhedral stratification.

If $P$ is any polyhedral cell, we define polyhedral functions $i_{P}$ and $j_P$ as follows:
$$
\begin{array}{c}
i_{P}(x) := 
{
\bigg\{
\begin{array}{ll}
1 & \text{if $x$ is in the closure of $P$} \\
0 & \text{otherwise}
\end{array}
}
\\
j_P(x)  :=  
{\bigg\{
\begin{array}{ll}
(-1)^{\dim P} & \text{if $x$ is in the interior of $P$}\\
0 & \text{if $x$ is outside of or on the boundary of $P$}
\end{array}
}
\end{array}
$$
We refer to $i_P$ and $j_P$ as the \emph{standard} and \emph{costandard} indicator functions associated to $P$.  (Note that $i_P$ is literally the indicator function of $\overline{P}$, not $P$ itself.  The names ``standard'' and ``costandard'' are motivated by the standard and costandard sheaves of \cite{NZ}.)  If $Q$ is the closure of a polyhedral cell $P$ we will set $i_Q = i_P$, $j_Q = j_P$.

\begin{proposition}
There is a unique linear map $\int:\PolF(V) \to \bZ$ that carries the indicator functions $j_P$ to 1. If $1_X$ denotes the indicator function of a polyhedral subset $X \subset V$, then we have
$$\int f = \chi_c(X)$$
where $\chi_c(X)$ denotes the alternating sum of the compactly-supported Betti numbers of $X$.
\end{proposition}

We will refer to $\int$ as ``integration with respect to Euler characteristic measure.''

If $f \in \PolF_c(V)$ and $X \subset V$ is a polyhedral subset, then we define
$$
\int_X f\vert_X := \int f \cdot 1_X
$$
where $1_X$ denotes the indicator function of $X \subset V$.
Given a linear map $u:V \to V'$, we may define operations $u^*:\PolF(V') \to \PolF(V)$ and $u_!:\PolF(V) \to \PolF(V')$ as follows:
$$
\begin{array}{rcl}
u^*(\phi)(x) & = &  \phi(u(x)) \\
u_!(\phi)(x) & = & \int_{u^{-1}(x)} \phi\vert_{u^{-1}(x)}
\end{array}
$$

\begin{remark}
The ring structure on $\PolF_c(M_\bR)$ that matches that on $K_T(X)$ under Morelli's map is the convolution product, given by the formula
$$f \star g := v_!(f \times g)$$
where $v:M_\bR \times M_\bR \to M_\bR$ denotes the addition map and $f \times g$ is the function that sends $(x,y)$ to $f(x)g(y)$.  Thus informally we have
$$f \star g (x) = \int f(y)g(x-y) d\chi_y$$
If $P$ and $Q$ are polytopes then the convolution of the functions $i_P$ and $i_Q$ is the standard indicator function on another polytope known as the Minkowski sum of $P$ and $Q$.
\end{remark}

Morelli \cite[2.5]{Mo2} defines a group of functions generated by ``spherical polyhedra.''  We will rephrase some of this theory in order to introduce a Fourier(-Sato) transform.  A polyhedral function $f:V \to \bR$ is called \emph{conical} if $f(\lambda \cdot x) = f(x)$ for all $\lambda \in \bR_{>0}$.  We write $\PolF_{\bR_{>0}}(V)$ for the group of conical functions on $V$.  

The germ of a polyhedral function around a point $x \in V$ is naturally identified with an element of $\PolF_{\bR_{>0}}(T_x V)$.  More precisely, there is a unique homomorphism $\nu_x:\PolF(V) \to \PolF_{\bR_{>0}}(T_x V)$ such that for all $f \in \PolF(V)$ there exists a neighborhood $U$ of $x$ with $\nu_x(f)\vert_U = f\vert_U$.  We will refer to the homomorphism $\nu_x$ as the \emph{specialization operator} at $x$.  A partner $\mu_x$ to the specialization operator, which we will refer to as the \emph{microlocalization operator}, can be described in terms of a kind of Fourier transform which we now define.

\begin{definition}
Let $V$ be a real vector space and let $V^*$ be its dual.  We define the \emph{Fourier-Sato transform} for functions to be the homomorphism
$$\FT:\PolF_{\bR_{> 0}}(V) \to \PolF_{\bR_{>0}}(V^*)$$
by the formula
$$\FT(f)(\xi) := \int_{\{v \in V \mid \langle\xi,v\rangle \leq 1\}} f$$
if $\xi \neq 0$ and $\FT(f)(0) := \int f$.
\end{definition}

\begin{remark}
This Fourier-Sato transform is closely related to a construction given by Barvinok \cite[Lecture 2, Theorem 4]{Ba}.
\end{remark}

\begin{example}
Suppose that $V = \bR$ and $V^* = \bR$ and the pairing is given by $\langle \xi,v\rangle = \xi v$.  If $f = i_{\{x \mid x \geq 0\}}$ is given by
$$f(x) = \bigg\{\begin{array}{cl}
1 & \text{if }x\geq 0\\
0 & \text{if }x < 0
\end{array}
$$
we have
\begin{itemize}
\item $\FT(f)(\xi) = \chi_c([0,1/\xi]) = 1$ is the compactly-supported Euler characteristic of a closed interval if $\xi$ is positive
\item $\FT(f)(\xi) = \chi_c([0,\infty)) = 0$ is the compactly supported Euler characteristic of a half-open interval if $\xi$ is negative or zero.
\end{itemize}
\end{example}

We have the following properties of the operator $\FT$:
$$\FT \circ \FT(f) (x) = f(-x)$$
for all $f \in \PolF_{\bR_{>0}}(V)$ and $x \in V$, and
$$\begin{array}{rcl}
\FT(i_\sigma) & = & (-1)^{\dim(\tau)} j_\tau\\
\FT(j_\tau) & = &i_{-\sigma}
\end{array}
$$
whenever $\sigma$ is a polyhedral cone in $V$, $\tau = \{\xi \in V^* \mid \forall x \in \sigma \xi(x) \geq 0\}$ is the dual cone in $V^*$, and $-\sigma$ denotes the image of $\sigma$ under the antipodal map.

We define a homomorphism $\mu_x:\PolF(V) \to \PolF(T_x^* V)$ by $\mu_x(f) := \FT(\nu_x(f))$.  If $\xi \in T_x^* V$ define $\mu_{x,\xi}(f) = \mu_x(f)(\xi)$.  For $f \in \PolF(V)$, we define the \emph{singular support} $\SS(f) \subset T^*V$ of $f$ to be the closure of the set of covectors $(x,\xi) \in V \times V^*$ such that $\mu_{x,\xi}(f) \neq 0$.  We have the following remarkable property:

\begin{theorem}
$\SS(f)$ is the closure of a Lagrangian submanifold of $T^*V$ with its natural symplectic structure.  In particular $\dim(\SS(f)) = \dim(V)$.
\end{theorem}

This can be proved in the polyhedral case by checking that it is true for standard functions $1_P$, and that $\SS(f + g) \subset \SS(f) \cup \SS(g)$.  

\begin{remark}
As $T^*V \cong  V \times V^*$, we can project the singular support $\SS(f)$ to $V^*$.  When $f$ is of the form $\Moi(\cE)$ for $\cE$ a toric vector bundle, the map $\SS(f) \to N_\bR$ is similar to Payne's ``branched cover of $\Sigma$'' constructed in \cite{P1}.
\end{remark}

\subsection{Construction of Morelli's map}

Morelli constructs the map $\Moi:K_T(X) \to \PolF(M_\bR)$ so that it obeys two rules
\begin{enumerate}
\item[(a)] If $\cE$ is a toric vector bundle, then for $\chi \in M_\bZ$
$$\Moi(\cE)(\chi) = \sum_{i = 0}^n (-1)^i \mathrm{rank}(H^i(X;\cE)_\chi)$$
where $H^i(X,\cE)_\chi$ denotes the $\chi$th weight space of the $T$-module $H^i(X,\cE)$.
\item[(b)] If $\Psi^k:K_T(X) \to K_T(X)$ denotes the $k$th Adams operation, then $\Moi(\Psi^k(\cE))(\chi) = \Moi(\cE)(\chi/k)$.
\end{enumerate}
These rules determines the function $\Moi(\cE)$ at elements of $M_\bQ = M_\bZ \otimes \bQ$, and can be extended to all of $M_\bR$ by a piecewise-continuity argument.  To verify that $\Moi$ is a well-defined isomorphism from this perspective takes some work.  We will here define Morelli's map more simply and directly, by a process that can also be done at the level of sheaves.

For each vector bundle $\cE$ and cone $\sigma$ of $\Sigma$, following \cite[3.1]{P1} we associate a multiset $u(\sigma,\cE) \subset M/(\sigma^\perp \cap M)$ as follows.  If $x \in X$ is a point in the $T$-orbit corresponding to $\sigma$, and $T_x \subset T$ is the isotropy subgroup, then $u(\sigma,\cE)$ is the collection of weights of the $T_x$ action on the fiber $\cE_x$.  Then we define
\begin{equation}
\label{eq:1}
\Moi(\cE) = \sum_{\sigma \in \Sigma} \sum_{\chi \in u(\sigma,\cE)} (-1)^{\dim(\sigma)} j_{\chi + \sigma^\vee}
\end{equation}
where $i_{\chi + \sigma^\vee}$ is the costandard indicator function of $\chi + \sigma^\vee \subset M_\bR$.  That this is additive for short exact sequences and satisfies Morelli's conditions (a) and (b) follows from a \v Cech theory argument.  

\begin{remark}
Note that we are expressing $\Moi(\cE)$ as an alternating sum of functions that are not in the image of $\Moi$.  Morally, the summands are the classes of certain equivariant \emph{quasi}coherent sheaves that $\cE$ can be expressed as a virtual sum of.  See Section \ref{sec:cech}.
\end{remark}

\begin{remark}
\label{rem:chern}
It follows from equation \ref{eq:1} and the formula for $\FT(j_{\sigma^\vee})$ that if $\sigma$ is a top-dimensional cone and $\xi$ belongs to the interior of $\sigma$, then $\mu_{x,-\xi} \Moi(\cE)$ is the multiplicity of $x$ in $u(\sigma,\cE)$.  Payne has shown \cite{P3} that the total equivariant Chern class $c_T(X) \in A_T^*(X)$---the equivariant Chow ring of $X$---can be computed from the multisets $u(\sigma,\cE)$.  This shows that $c_T(\cE)$ can be computed from $\Moi(\cE)$ using the microlocal operators $\mu_{x,-\xi}$.
\end{remark}

We can use the language of Section \ref{sec:ecm} to state a sharper version of Theorem \ref{thm:morelli1}.  First, if $\Sigma$ is a complete rational polyhedral fan in $N_\bR$, let $\cH_\Sigma$ be the periodic affine hyperplane arrangement in $M_\bR$ obtained from translating the hyperplanes perpendicular to the rays of $\Sigma$ by all lattice points $M_\bZ$.  This hyperplane arrangement cuts $M_\bR$ into polyhedral cells, and we write $\PolF(M_\bR;\cH_\Sigma) \subset \PolF(M_\bR)$ for the subgroup of functions that are constructible along $\cH_\Sigma$.  Then Morelli's theorem is as follows

\begin{theorem}[{\cite[Theorem 8]{Mo}}]
\label{thm:morelli2}
Let $X$ be a complete toric variety with fan $\Sigma$.  The following are isomorphic
\begin{enumerate}
\item $K_T(X)$.
\item The subgroup of $\PolF(M_\bR;\cH_\Sigma)$ of functions $f$ satisfying the following condition: at each point $x \in M_\bR$, the germ $\nu_x f$ can be written as a linear combination of functions of the form $j_{\sigma^\vee}$, where $\sigma$ runs through $\Sigma$.
\end{enumerate}
\end{theorem}

We can rephrase Morelli's condition (2) in terms of $\mu_x$ and singular support.  Define a subset $\Lambda_\Sigma \subset T^* M_\bR$ by the formula
\begin{equation}
\label{eq:2}
\LS = \bigcup_{\tau\in\Sigma} (\tau^\perp + M)\times -\tau
\end{equation}

\begin{corollary}
\label{cor:morelli3}
Let $X$ be a complete toric variety with fan $\Sigma$.  Then (1) and (2) of Theorem \ref{thm:morelli2} are also isomorphic to the following:
\begin{enumerate}
\item[(3)] The subgroup of $\PolF(M_\bR;\cH)$ consisting of functions $f$ with the property that at each point $x \in M_\bR$, the function $\mu_x(f):N_\bR \to \bZ$ is constant along the interior of each cone of the fan $-\Sigma$ antipodal to $\Sigma$.
\item[(4)] The group $\PolF(M_\bR;\LS)$ of functions $f$ satisfying $\SS(f) \subset \LS$.
\end{enumerate}
\end{corollary}

\begin{proof}
The groups (2), (3), and (4) are all subgroups of $\PolF(M_\bR)$.  That (2) and (3) coincide follows from the fact that $\nu_x$ and $\mu_x$ differ by a Fourier-Sato transform, which is an isomorphism, and from the fact that the Fourier transform of a costandard indicator function associated to $\sigma^\vee$ is a standard indicator function associated to $-\sigma$.  To see that (3) and (4) coincide, note that $f$ is constructible with respect to $\cH$ if and only if $\SS(f)$ is contained in the \emph{characteristic variety} $\Lambda_\cH$ of the stratification $\cH_\Sigma$
$$\Lambda_\cH:= \bigcup_{s \text{ a cell of } \cH} T^*_s V$$
and that $\Lambda_\cH$ contains $\LS$.
\end{proof}

\section{Review of the coherent-constructible correspondence (CCC)}
\label{sec:CCCrev}

A basic consequence of the coherent-constructible correspondence is the following:

\begin{theorem}
For every equivariant vector bundle $\cE$ on $X$, there is a complex $\kappa(\cE)$ of sheaves on $M_\bR$, together with natural isomorphisms
$$\Hom(\cE,\cF) \cong \Hom(\kappa(\cE),\kappa(\cF))$$
for any pair of equivariant vector bundles.  The homs on the left are taken in the category of equivariant vector bundles, and the homs on the right are taken in the homotopy category of complexes of sheaves.
\end{theorem}

The complete result says more.  For instance, the higher Ext groups between vector bundles $\cE$ and $\cF$ are naturally isomorhpic to the Ext groups (or what are sometimes called hyper-Ext groups) between the complexes $\kappa(\cE)$ and $\kappa(\cF)$.  A more modern formulation of this is that there is a fully faithful embedding of derived categories, from the derived category of equivariant coherent sheaves on $X$ (or perfect complexes if $X$ is singular) to the derived category of constructible sheaves on $M_\bR$.  The main result of \cite{fltz} is a characterization of which constructible sheaves can appear, similar to Morelli's characterization \ref{thm:morelli2} of the image of $\mathit{mo}$.

Although this result gives us a good understanding of which complexes of constructible sheaves on $X$ come from \emph{complexes} of vector bundles, it does not say which ones come from genuine vector bundles, i.e. complexes concentrated in a single homological degree.  We turn to this question in sections \ref{sec:four} and \ref{sec:five}.

\subsection{$\Theta$-complexes} In this section we define the functor $\kappa$.  It is a categorical version of the definition of $\mathit{mo}$ given in equation \ref{eq:1}.

Let $\bGamma(M_\bR,\Sigma)$ denote the set of pairs $(\sigma,\chi)$ where $\sigma$ is a cone in $\Sigma$ and $\chi \subset M_\bR$ is an integral coset of $\sigma^\perp$.  (``Integral'' means that $\chi$ passes through a lattice point.)  The set $\bGamma(M_\bR,\Sigma)$ indexes a class of constructible sheaves on $M_\bR$ and a class of equivariant quasicoherent sheaves on $X$.

\begin{definition}[$\Theta(\sigma,\chi)$]
For each $(\sigma,\chi) \in \bGamma(\Sigma,M)$ we have the open set $(\chi + \sigma^\vee)^\circ \subset M_\bR$, consisting of all those $v \in M_\bR$ with $\langle v,-\rangle > \langle \chi,-\rangle$ when evaluated at elements of $\sigma$.  The object $\Theta(\sigma,\chi)$ is the extension-by-zero of the rank one constant sheaf $\sfR$ on$(\chi + \sigma^\vee)^\circ$.
\end{definition}

\begin{remark}
This differs from the definition of $\Theta(\sigma,\chi)$ given in \cite{fltz}: there, we put the Verdier dualizing sheaf on the closed set $\chi + \sigma^\vee$ and extended by zero.  Choosing an orientation of $M_\bR$ identifies (up to a shift) these two veresions of $\Theta(\sigma,\chi)$.
\end{remark}

\begin{definition}
For each $(\sigma,\chi) \in \bGamma(\Sigma,M)$, let $\Theta'(\sigma,\chi) = j_*(\cO_{U_\sigma}(\chi))$ where $j$ is the inclusion of $U_\sigma$ into $X$ and $\cO_{U_\sigma}(\chi)$ denotes the structure sheaf on $U_\sigma$ endowed with an equivariant structure $\chi$.
\end{definition}

The set $\bGamma(\Sigma,M)$ has a partial order: we say $(\sigma,\chi) \leq (\tau,\psi)$ whenever $\tau \subset \sigma$ and $\phi-\psi \in \tau^\vee$.

\begin{proposition}[{\cite[Proposition 3.3]{fltz}}]
We have canonical isomorphisms
$$\begin{array}{ccc}
\Ext^i(\Theta(\sigma,\phi),\Theta(\tau,\psi)) & = & \bigg\{ 
\begin{array}{cc}
\sfR & \text{if $(\sigma,\phi) \leq (\tau,\psi)$ and $i = 0$}\\
0 & \text{ otherwise}
\end{array} \\
\Ext^i(\Theta'(\sigma,\phi),\Theta'(\tau,\psi)) & = & \bigg\{ 
\begin{array}{cc}
\sfR & \text{if $(\sigma,\phi) \leq (\tau,\psi)$ and $i = 0$}\\
0 & \text{ otherwise}
\end{array} \\
\end{array}
$$
\end{proposition}

A consequence of the proposition is that we can model the derived category $D^b(\cQ_T(X))$ and $D^b(\Sh(M_\bR))$---at least, full subcategories of them---by complexes of the objects $\Theta$.

\begin{definition}
A $\Theta$-complex (resp. $\Theta'$-complex) is a bounded complex of sheaves on $M_\bR$ each term of which is isomorphic to something of the form $\bigoplus_{i = 1}^k \Theta(\sigma_i,\chi_i)$ (resp $\bigoplus_{i = 1}^k \Theta'(\sigma_i,\chi_i)$).  We let $\ltr$ (resp. $\ltrp$) denote the category of $\Theta$-complexes and chain maps (resp. $\Theta'$-complexes and chain maps) and we let $h\ltr$ and $h\ltrp$ denote the categories with the same objects whose homomorphisms are chain homotopy classes of maps.
\end{definition}

\begin{remark}
The notation $\ltr$ and $\ltrp$ means something slightly different in \cite{fltz}---a dg enrichment of the categories of chain complexes we are considering here.
\end{remark}

\begin{theorem}[{\cite[Theorem 3.4]{fltz}}]
\begin{enumerate}
\item
The tautological functors $\ltrp \to D^b(\cQ_T(X))$ and $\ltr \to D^b(\Sh(M_\bR)$ induce full embeddings of triangulated categories $h\ltrp \hookrightarrow D^b(\cQ_T(X))$ and $h\ltr \hookrightarrow D^b(\Sh(M_\bR))$.
\item Up to unique natural isomorphism, there is a unique equivalence of categories $\ltr \cong \ltrp$ that carries $\Theta'(\sigma,\chi)$ to $\Theta(\sigma,\chi)$ and that for $(\sigma,\phi) \leq (\tau,\psi)$ carries the inclusion map $\Theta'(\sigma,\phi)$ to the inclusion map $\Theta(\sigma,\phi) \to \Theta(\tau,\psi)$.
\end{enumerate}
\end{theorem}

\begin{definition}
Let $\kappa:h\ltrp \to h\ltr$ denote the shift-by-dim($M_\bR)$ of the equivalence of the theorem. i.e. $\kappa(\Theta'(\sigma,\chi)) := \Theta(\sigma,\chi)[\dim(M_\bR)]$.
\end{definition}

When comparing coherent constructions to constructible ones, we will frequently use the following fact:

\begin{proposition}
\label{prop:nat}
Let $u$ and $v$ be two exact functors from $h\ltrp$ to a second triangulated category $D$.  Any natural transformation $n:u \to v$ is determined by its values on the objects $\Theta'(\sigma,\chi)$, and conversely any system of maps $n_{\sigma,\chi}:u(\Theta'(\sigma,\chi)) \to v(\Theta'(\sigma,\chi))$ in $D$ that make the squares
$$
\xymatrix{
u(\Theta'(\sigma,\chi)) \ar[r] \ar[d] & u(\Theta'(\tau,\psi)) \ar[d]\\
v(\Theta'(\sigma,\chi)) \ar[r] & v(\Theta'(\tau,\psi))
}
$$
commute extends to a unique natural transformation $n:u \to v$.  The natural transformation is an isomorphism if and only if each $n_{\sigma,\chi}$ is an isomorphism.
\end{proposition}

\subsection{Vector bundles, perfect complexes, and $\Theta'$-complexes}
\label{sec:cech}

Let $\cE$ be a vector bundle on $X$.  For each $\sigma \in \Sigma$ we may form the quasicoherent sheaf $j_{\sigma *} j_\sigma^* \cE$, which is necessarily of the form $\bigoplus_{i = 1}^n \Theta'(\sigma,\chi_i)$.  These sheaves naturally form a \v Cech complex quasi-isomorphic to $\cE$, and in this way we can construct a functor from the category of $T$-equivariant vector bundles to $h\ltrp$.  (Or in fact to $\ltrp$.)  We make the following definitions:

\begin{itemize}
\item $\vC^\sigma(\cE)$ is the quasicoherent sheaf $j_{\sigma*}j_\sigma^* \cE$.  If $\tau \subset \sigma$ then we let $\can_{\sigma,\tau}$ denote the canonical adjunction map $j_{\sigma*} j_\sigma^* \cE \to j_{\tau *} j_\tau^* \cE$. 
\item $\vC^k(\cE) = \bigoplus_{\sigma \in \Sigma(\dim(M_\bR)-k)} \vC^\sigma(\cE)$
\end{itemize}
We wish to assemble the $\vC^k$ into a cochain complex $\vC^\bullet$.  A differential $d:\vC^k \to \vC^{k+1}$ may be given by describing its matrix entries $d_{\sigma,\tau}:\vC^\sigma \to \vC^\tau$, where $\sigma$ runs through $k$-codimensional cones and $\tau$ through $\sigma$-codimensional cones.  \begin{itemize}
\item Pick, once and for all, an orientation of each cone in $\Sigma$.  $\vC^\bullet(\cE)$ is the cochain complex whose degree $k$ piece is $\vC^k(\cE)$ and whose differential has matrix entries
 $$d_{\sigma,\tau} = \Bigg\{ \begin{array}{cl}
\can_{\sigma,\tau} & \text{if $\tau \subset \sigma$ and the orientations of $\tau$ and $\sigma$ agree}\\
-\can_{\sigma,\tau} & \text{if $\tau \subset \sigma$ and the orientations of $\tau$ and $\sigma$ disagree}\\
0 & \text{if $\tau$ is not contained $\sigma$}
\end{array}
 $$
\end{itemize}

A standard \v Cech theory argument shows that $\vC^\bullet(\cE)$ is a resolution of $\cE$.  We will abuse notation and write $\kappa(\cE) = \kappa(\vC^\bullet(\cE))$.

\begin{example}
\label{ex:taut}
If $\cL$ is an equivariant ample line bundle on a projective toric variety, then the associated moment polytope is the convex hull in $M_\bR$ of the weights of the $T$-action on the fibers of $\cL$ over fixed points.  There is a similar polytope associated to an anti-ample line bundle (i.e. a line bundle whose dual is ample.)  It is proved in \cite{fltz} that $\kappa$ of an anti-ample line bundle is quasi-isomorphic to the ``standard'', i.e. constant, sheaf on the associated polytope, and that $\kappa$ of an ample line bundle is quasi-isomorphic to the ``costandard'' sheaf (or complex of sheaves) on the interior of the associated polytope.  (If $U$ is an open subset of an oriented $n$-dimensional manifold, then the costandard sheaf is the extension-by-zero of the constant sheaf on $U$, placed in degree $-n$.)  For example, if $\cO(-1)$ and $\cO(1)$ are the tautological and anti-tautological bundles on $\mathbb{P}^2$, endowed with equivariant structures, then the sheaves $\kappa(\cO(1))$ and $\kappa(\cO(-1))$ are as in the diagram
\begin{center}
\includegraphics[scale=.6]{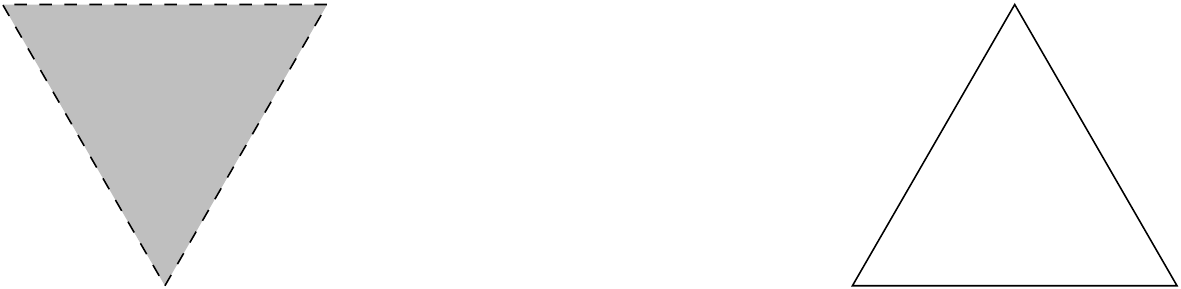}
\end{center}
The left figure is $\kappa(\cO(1))$ and the right figure is $\kappa(\cO(-1))$.  The stalks along the dotted boundary are zero.  The darkly shaded region is a rank one constant sheaf placed in cohomological degree $-2$, and the unshaded region is a rank one constant sheaf placed in cohomological degree zero.
\end{example}

\begin{example}
\label{ex:hirz}
The figure shows the constructible sheaves associated to several equivariant line bundles on the Hirzebruch surface $F_1$ (i.e., on the blowup of $\mathbb{P}^2$ at a $T$-fixed point.)
\begin{center}
\includegraphics[scale=.6]{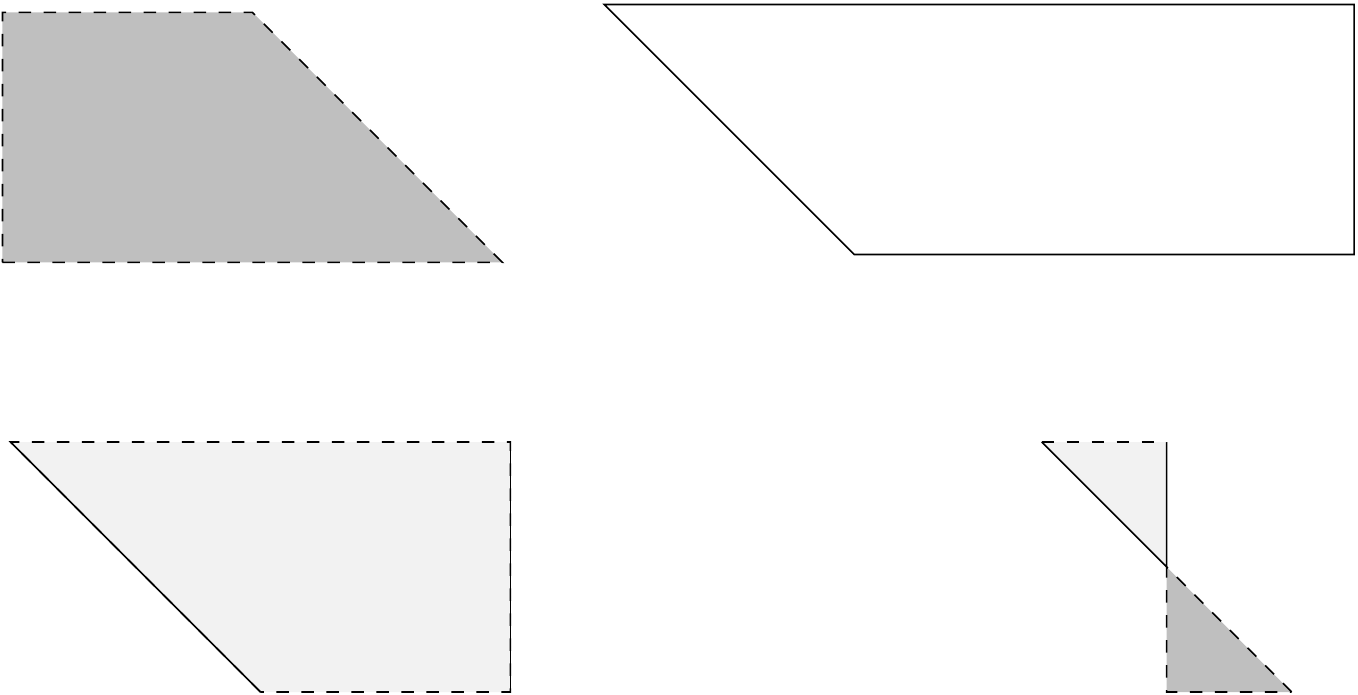}
\end{center}
The top row shows sheaves associated to ample and anti-ample line bundles.  The bottom row shows line bundles that are neither ample nor anti-ample.  The meanings of the dotted lines and shadings are the same as in the previous example; the lighter shade indicates that the sheaf is in cohomological degree $-1$.  For a discussion of how the lower right sheaf behaves in a neighborhood of the point where the two triangles meet (a point which does not necessarily belong to $M_\bZ$), see Example \ref{ex:weirdlocal}.
\end{example}

\subsection{Generic stalks and global sections}

\begin{proposition}
\label{prop:rgaacyc}
Let $\sigma \subset N_\bR$ be a polyhedral cone and let $c \subset M_\bR$ be a coset of $\sigma^\perp$.  Then if we pick an orientation of $M_\bR$ we have canonical isomorphisms
$$
\bfR^i\Gamma_c(\Theta(\sigma,\chi)[\dim(M_\bR)]) \cong \bigg\{
\begin{array}{cc}
\sfR & \text{if $i = 0$} \\
0 & \text{otherwise}
\end{array}
$$
\end{proposition}

\begin{proof}
For any sheaf of the form $j_! F$, where $j:A \hookrightarrow X$ is a locally closed inclusion, there is a natural isomorphism $\RGa_c(X,j_!F) \cong \RGa_c(A,F)$.  Thus $\bfR^i\Gamma_c(M_\bR;\Theta(\sigma,\chi)) \cong \bfR^i\Gamma_c((\chi + \sigma^\vee)^\circ;\sfR) \cong H^i_c((\chi + \sigma^\vee)^\circ)$.  Since $(\chi + \sigma^\vee)$ is an oriented contractible open manifold of dimension $\dim(M_\bR)$, this group is $\sfR$ if $i = \dim(M_\bR)$, and $0$ otherwise.  The isomorphism $\bfR^i\Gamma_c(\Theta(\sigma,\chi)[\dim(M_\bR)]) \cong \bfR^{i + \dim(M_\bR)}\Gamma_c(\Theta(\sigma,\chi))$ completes the proof.
\end{proof}

As a consequence, when $F^\bullet$ is a $\Theta$-complex, we may compute $\RGa_c(F^\bullet)$ by applying $\Gamma_c$ term by term.

\begin{theorem}
\label{thm:fiber}
Let $\cE$ be an equivariant perfect complex on $X$ and let $x_0$ be a point in the open orbit of $X$.  Then if we pick an orientation of $M_\bR$ there is a natural quasi-isomorphism
$$\cE_{x_0} \cong \RGa_c(\kappa(\cE))$$
\end{theorem}

\begin{proof}
It follows from Propositions \ref{prop:nat} and \ref{prop:rgaacyc}.
\end{proof}

In particular, if $\cE$ is a vector bundle rather than a complex of vector bundles, then $\RGa_c(\kappa(\cE))$ is concentrated in one degree.

\section{Morse and Klyachko filtrations}
\label{sec:four}

In this section we prove Theorems \ref{thm:1.1} and \ref{thm:1.1.5}.

\subsection{Klyachko's filtrations}
Let $G$ be an algebraic group and let $X$ be a variety equipped with a $G$-action.  Let $\pi:\cE \to X$ be a vector bundle on $X$.   A $G$-equivariant structure on $\cE$ is a $G$-action on the total space of $\cE$ that's compatible with $\pi$ and the action on $X$.

An important example is when $X = \bA^1$ and $G = \Gm$.    Vector bundles on $X$ are all trivial: a vector bundle $\cE$ can be recovered naturally from the fiber $\cE_1$ of $\cE$ at $1 \in \bA_1$.  An equivariant structure on $\cE$ induces an interesting structure on $\cE_1$: a filtration.  The $i$th filtered piece of $\cE_1$ is given by those fibers that extend to global sections of weight $i$ (under the $\Gm$-action on $\Gamma(\cE)$).  If $\Gm$ acts on $\bA^1$ via $(\lambda, z) \mapsto \lambda \cdot z$, then this is an decreasing filtration, and if it acts via $(\lambda,z)\mapsto \lambda^{-1} \cdot z$ it is an increasing filtration.  We prefer the latter convention.

Let us spell this out more precisely.

\begin{definition}
If $\cE$ is an equivariant vector bundle on $\bA^1$ (with $\Gm$ acting via $\lambda^{-1} \cdot z$) then we introduce the following notation:
\begin{itemize}
\item $\Gamma(\cE)$ denotes the global sections of $\cE$, together with its $\Gm$-action by
$$(\lambda \star s)(z) = \lambda \star (s(\lambda^{-1} \star z)) = \lambda \star (s(\lambda \cdot z))$$
\item $\Gamma(\cE)_i$ denotes the $i$th weight space, i.e.
$$\Gamma(\cE)_i = \{s \mid \lambda \star s = \lambda^i \cdot s\}$$
\end{itemize} 
Note that if $s \in \Gamma(\cE)_i$ then $z \cdot s \in \Gamma(\cE)_{i+1}$.  
\begin{itemize}
\item We let $F_{\leq i} (\cE_1)$ be the image of the (injective) restriction map $\Gamma(\cE)_i \to \cE_1$
\end{itemize}
\end{definition}

\begin{theorem}
The assignment $\cE \mapsto (\cE_1,F_{\leq})$ is an equivalence between the category of $\Gm$-equivariant vector bundles on $\bA^1$ and vector spaces equipped with increasing filtrations.  Under this correspondence, the $i$th piece of the associated graded vector space is naturally isomorphic to the $i$th weight space of the $\Gm$-action on the fiber $\cE_0$. 
\end{theorem}

Klyachko \cite{K} used this theorem to classify torus-equivariant vector bundles on toric varieties.  Let $X$ be a toric variety with a $T$-action, and let $x_0 \in X$ be a base point in the open orbit.  If $\alpha \in N = \Hom(\Gm,T)$ generates a ray in the fan defining $X$, then there is an $\alpha$-equivariant map $i_\alpha: \bA^1 \to X$ that carries $1$ to $x_0$ and $0$ to a point in the $T$-stable divisor corresponding to $\alpha$.  If $\cE$ is an equivariant vector bundle on $X$ we can pull it back via this map to obtain a $\Gm$-equivariant vector bundle on $\bA^1$.  By the previous theorem we get a filtration on the vector space $E = \cE_{x_0}$, which we denote by $E^\alpha$.  More precisely we define
$$E^{\alpha}_{\leq k} = F_{\leq k} (i_\alpha^* \cE)$$
Klyachko observed that these filtrations satisfy the following compatibility condition

\vspace{.1in}

(C)\quad Let $\sigma \in \Sigma$, and let $\alpha_1,\ldots, \alpha_k$ be the generators of the 1-dimensional rays in $\sigma$.  Then the subspaces $E^{\alpha_j}_{\leq i}$ consist of coordinate subspaces of some basis of $E$.
\vspace{.1in}

\begin{theorem}[Klyachko]
The assignment $\cE \mapsto (E,\{E_\alpha\})$ is an equivalence between the category of equivariant vector bundles and the category of vector spaces equipped with filtrations satisfying condition (C).
\end{theorem}

\begin{remark}
Klyachko uses a different convention and works with decreasing filtrations.  His $E^\alpha(i)$ is our $E^{\alpha}_{\leq -i}$.
\end{remark}

\subsection{Morse filtrations of $\Theta$-complexes}
\label{sec:morsetheta}

If $U$ is an open subset of $X$, and $F$ is a sheaf on $X$, then there is an inclusion $\Gamma_c(U;F\vert_U) \hookrightarrow \Gamma_c(X;F)$.   If $U$ is of the form $\{x \mid f(x) < t\}$ for some continuous function $f:X \to \bR$ and $t \in \bR$, let us put $\Gamma_{c,f < t}(F) = \Gamma_c(U;F\vert_U)$. This is a left-exact functor and we let $\RGa_{c,f < t}$ denote its right derived functor.  

\begin{proposition}
\label{prop:morsetheta}
Let $\sigma \subset N_\bR$ be a polyhedral cone and let $\chi \subset M_\bR$ be a coset of $\sigma^\perp$.  Suppose $f:M_\bR \to \bR$ is a convex continuous function, i.e. for each $t$ the set $\{f < t\} = \{x \in M_\bR \mid f(x) < t\}$ is convex.  Then if we pick an orientation of $M_\bR$ we have naturally
$$\bfR^i\Gamma_{c,f < t}(\Theta(\sigma,\chi)[\dim(M_\bR)])\cong \bigg\{
\begin{array}{cl}
\sfR & \text{ if $i = 0$ and $\{f < t\} \cap (\chi+ \sigma^\vee)^\circ \neq \varnothing$}\\
0 & \text{ otherwise}
\end{array} 
$$
\end{proposition}

\begin{proof}
By definition, 
$$\begin{array}{rl}
\bfR^i\Gamma_{c,f < t}(\Theta(\sigma,\chi)[\dim(M_\bR)]) & = \bfR^{i + \dim(M_\bR)}\Gamma_c(\{f < t\};\Theta(\sigma,\chi)\vert_{\{f < t\}}) \\
& = H^{i + \dim(M_\bR)}_c((\chi + \sigma^\vee)^\circ \cap \{f < t\})
\end{array}
$$
The proposition follows from the fact that both $(\chi + \sigma^\vee)^\circ$ and $\{f < t\}$ are convex open subsets of $M_\bR$.
\end{proof}

As a consequence, when $F^\bullet$ is a $\Theta$-complex, we may compute $\RGa_{c,f < t}(F^\bullet)$ term by term.  The result is a subcomplex of $\RGa_c(F^\bullet)$, computed according to Proposition \ref{prop:rgaacyc} by applying $\Ga_c$ term by term.  Letting $t$ vary this gives us an open increasing filtration (see \ref{sec:filt}) of the complex $\RGa_c(F^\bullet)$.  We refer to this as the \emph{Morse filtration} of $F^\bullet$ associated to the function $f$.  The proposition says that the Morse filtration of $\RGa_c(\Theta(\sigma,\chi))$ is pure of weight $\inf_{x \in (\chi+ \sigma^\vee)^\circ} f(x)$.

\subsection{Proof of Theorems \ref{thm:1.1} and \ref{thm:1.1.5}}

We will only work with the Morse filtrations of Section \ref{sec:morsetheta} associated to linear functions, which we may identify with elements of $N_\bR$.

Let us record the following corollary of Proposition \ref{prop:morsetheta}

\begin{lemma}
\label{lem:morsetheta}
Let $\cE$ be an equivariant vector bundle on $X$, and suppose that $f \in N_\bR$ belongs to $\tau^\circ$.  Then the Morse filtrations of $\RGa_c(\kappa(\cE))$ and of $\RGa_c(\kappa(\cE\vert_{U_\tau}))$ with respect to $f$ agree.
\end{lemma}

\begin{proof}
We have by definition $\kappa(\cE) = \kappa(\vC^\bullet(\cE))$.  Since $\Theta'(\sigma,\chi) \vert_{U_\tau} = \Theta'(\sigma \cap \tau,\chi + (\sigma \cap \tau)^\perp)$, we may compute $\kappa(\cE\vert_{U_\tau})$ by replacing each $\Theta(\sigma,\chi)$ in the $\Theta$-complex $\kappa(\vC^\bullet)$ by $\Theta(\sigma \cap \tau,\chi+ (\sigma \cap \tau)^\perp)$.  To show that the filtrations of $\RGa_c(\kappa(\cE))$ and $\RGa_c(\kappa(\cE\vert_{U_\tau}))$ agree it therefore suffices to show that the filtrations of $\RGa_c(\Theta(\sigma,\chi))$ and $\RGa_c(\Theta(\sigma \cap \tau,\chi + (\sigma \cap \tau)^\perp)$ agree.  This follows from Proposition \ref{prop:morsetheta}, and from the fact if $f \in \tau^\circ$ then the minimum value $f$ takes on $(\sigma \cap \tau)^\vee$ is the same as the minimum value it takes on $\sigma^\vee$.  
\end{proof}

If $C^\bullet$ has cohomology in only one degree then let us call a filtration of $C^\bullet$ \emph{strict} if each filtered piece has cohomology concentrated in the same degree.  Thus Theorem \ref{thm:1.1} says that the Morse filtration $\RGa_{c,f<t})$ of $\RGa_c(\kappa(\cE))$ is strict for every vector bundle $\cE$ and $f \in N_\bR$.

\begin{proof}[Proof of Theorem \ref{thm:1.1}]
The fact that $\RGa_c(F^\bullet)$ is concentrated in degree zero is proved in Theorem \ref{thm:fiber}.  To prove that the filtration is strict we have to show that $\bfR^i\Ga_{c,f< t}(F^\bullet)$ vanishes for $i > 0$.  By Lemma \ref{lem:morsetheta}, we may assume that $X = \Spec \sfR[\tau^\vee]$ and that $f$ belongs to $\tau^\circ$.  Since $X$ is affine, $\cE$ splits as a sum of equivariant line bundles $\cE = \cO_X(\chi_1) \oplus \cdots \oplus \cO_X(\chi_r)$ in which case $\kappa(\cE) = \Theta(\tau,\chi_1) \oplus \cdots \Theta(\tau,\chi_r)$.  The fact that $\bfR^i\Ga_{c,f < t}(F^\bullet)$ vanishes for $i > 0$ now follows from Proposition \ref{prop:morsetheta}.
\end{proof}

\begin{proof}[Proof of Theorem \ref{thm:1.1.5}]
We have seen in Theorem \ref{thm:fiber} that there is a natural isomorphism $H^0(M_\bR;\kappa(\cE)) \cong E$.  Let us abuse notation and let $\alpha$ denote both the generator of a ray in $\Sigma$ and the ray itself $\alpha = \bR_{>0} \cdot \alpha$.  The Morse filtration of $\RGa_c(\kappa(\cE))$ coincides with the Morse filtration of $\RGa_c(\kappa(\cE\vert_{U_{\alpha}}))$.  By definition, the Klyachko filtration of the generic fiber of $\cE$ coincides with the Klyachko filtration of the generic fiber of $\cE\vert_{U_{\alpha}}$.  We are therefore reduced to proving that the filtrations coincide in case $X = U_{\alpha}$.  As in the previous proof we use the fact that vector bundles on an affine $X$ split to reduce further to case when $\kappa(\cE) = \Theta(\bR \cdot \alpha,\chi)[\dim(M_\bR)]$.  In that case both the Morse filtration of $\RGa_c(\kappa(\cE))$ and the Klyachko filtration of the generic fiber of $\cO(c)$ are pure of weight $\langle \alpha,\chi\rangle$.  This completes the proof.
\end{proof}

\section{More Morse theory: microlocal sheaf theory}
\label{sec:five}

The Morse filtration defined in Section \ref{sec:morsetheta} is part of  a more general story of microlocal sheaf theory developed by Kashiwara and Schapira.  We will indicate how parts of this theory work for a special class of sheaves (polyhedral sheaves on real vector spaces) and what the consequences are for the CCC.  Most of the operations described in this section are sheaf-theoretic counterparts of operations on functions described in Section \ref{sec:ecm}.

\subsection{Polyhedral sheaves}
\label{sec:polyh}

Let $V$ be a real vector space, and recall from Section \ref{sec:ecm} the definition of a polyhedral stratification of $V$.  A \emph{polyhedral sheaf} on $V$ is a sheaf that is constant along the cells in a polyhedral stratification.  If $F$ is constant along the cells in $S$ then we will say that $F$ is $S$-polyhedral.  A theorem of Allen Shepard states that the inclusion functor from the derived category of the abelian category of $S$-polyhedral sheaves into the constructible derived category $D^b_c(V)$ is full \cite{Shep}.  We call objects of the essential image of this embedding $S$-polyhedral complexes---they are exactly those bounded complexes of sheaves whose cohomology sheaves are $S$-polyhedral.

To give an $S$-polyhedral sheaf is equivalent to giving the following data.
\begin{enumerate}
\item A vector space $F_a$ for each cell $a$ of $S$
\item A ``restriction map'' $r_{ab}:F_a \to F_b$ for every pair of cells with $a \subset \overline{b}$
\end{enumerate}
The restriction maps are moreover required to make all triangles commute: if $a \subset \overline{b}$ and $b \subset \overline{c}$, then $r_{bc} \circ r_{ab} = r_{ac}$.  Let us call the collection $(F_a,r_{ab})$ an $S$-combinatorial sheaf.  To extract an $S$-combinatorial sheaf from an $S$-polyhedral sheaf, we define the \emph{star} of $a \in S$ to be the union of cells belonging to $S$ and containing $a$ in their closure.  We set $F_a := \Gamma(\text{star}(a);F)$ and take $r_{ab}$ to be the restriction map associated to the inclusion $\text{star}(b) \subset \text{star}(a)$.  This operation is an equivalence of categories.  Shepard's theorem implies that to give an object of $D^b_S(V)$ is equivalent to giving a complex of $S$-combinatorial sheaves.

\begin{remark}
If $F$ is $S$-constructible, and $a\in S$ then for each point $x \in a$, the restrict-to-the-stalk map $\Gamma(\text{star}(a);F) \to F_x$ is an isomorphism.  Thus the values $F_a$ of an $S$-combinatorial sheaf are stalks of the corresponding $S$-constructible sheaf.
\end{remark}

\subsection{Morse groups and microlocal stalks}
\label{sec:mgams}

Let $V$ be a manifold and $f:V \to \bR$ a smooth function.  For each sheaf $F$ on $V$ define $\Gamma_{f \leq t}(F)$ to be the subgroup of $\Gamma(F)$ consisting of sections supported in the closed set $\{f \leq t\} := \{v \in V \mid f(v) \leq t\}$.  If $F^\bullet$ is a complex of flasque sheaves on $V$ then we can apply $\Gamma_{f \leq t}$ term-by-term to obtain a subcomplex $\RGa_{f \leq t}(F^\bullet) \subset \RGa(F^\bullet)$.  In this way we get a closed $\bR$-indexed filtration of $\RGa(F^\bullet)$.  

\begin{remark}
If $V = M_\bR$ and $F^\bullet$ is a $\Theta$-complex with compact support, then we have naturally $\RGa(F^\bullet) \cong \RGa_c(F^\bullet)$.  By definition, the Morse filtration $\RGa_{c,f<t}(F^\bullet)$ considered in Section \ref{sec:morsetheta} of this complex is the open counterpart of the closed Morse filtration $\RGa_{f\leq t}(F^\bullet)$ just defined.
\end{remark}

We can form a local version of this Morse filtration around a point $x \in V$: we can restrict $F^\bullet$ and $f$ to an $\epsilon$-ball $U$ around  $x$ and form $\RGa_{f \leq t}(U;F^\bullet) \subset \RGa(U;F^\bullet)$.  This filtration is not indexed by all of $\bR$ but by an interval around $f(x)$ whose length shrinks with $\epsilon$.  If $F^\bullet$ is cohomologically constructible, then the germ of this filtration is independent of $U$.  In particular there is a well-defined ``$f(x)$-graded piece'' of this filtered complex, yielding an object $\Mo_{x,-f}(F^\bullet) \in D^b(\Rmod)$ which we call the \emph{Morse group}.

\begin{remark}
An excision argument shows that the cohomology groups of the complex $\Mo_{x,-f}(F^\bullet)$ are naturally isomorphic to the relative groups $H^i(U,\{f > f(x)+\delta\};F^\bullet)$ for $U$ and $\delta$ suitably small.  (The Morse groups are typically defined as relative groups of the form $H^i(U,\{g < g(x) - \delta\};-)$; we have included a minus sign in the subscript to make our definition consistent with this.)  The figure below illustrates this latter definition
\begin{center}
\includegraphics[scale=.33,angle=45]{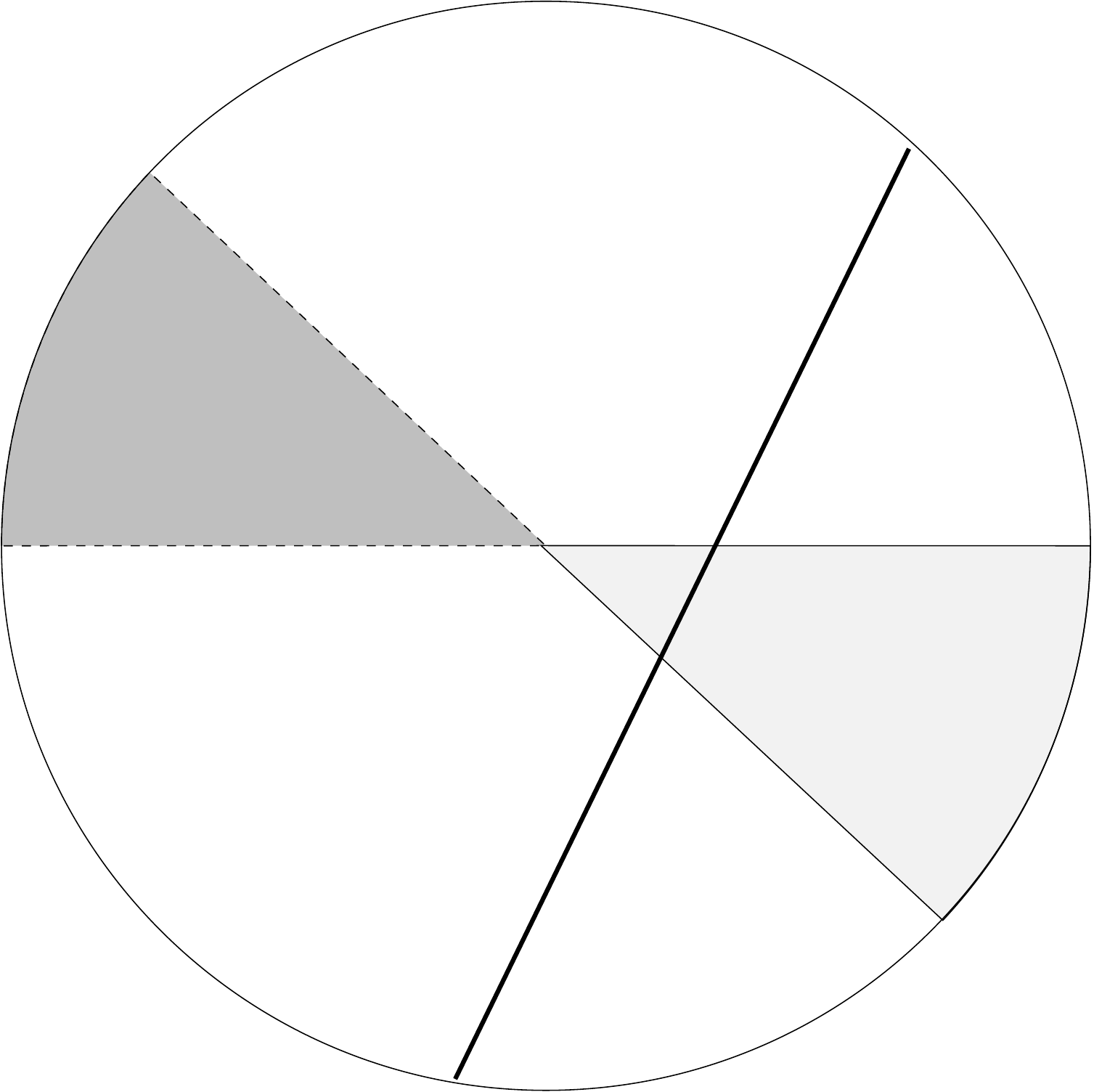}
\end{center}
The area within the circle is the neighborhood $U$, the central point is $x$, the shaded region is the sheaf $F^\bullet$ restricted to $U$, and the remaining line is the level set of $f$ at $f(x) + \delta$.
\end{remark}

Now suppose $V$ is a real vector space and $F^\bullet$ is a polyhedral sheaf on $V$, and let $(x,\xi) \in V \times V^* = T^* V$ be a cotangent vector.  We may regard $\xi$ as a linear function on $V$ and define the \emph{microlocal stalk} $\mu_{x,\xi}(F^\bullet)$ by the formula
$$\mu_{x,\xi}(F^\bullet) = \Mo_{x,\xi}(F^\bullet)$$
where $\xi$ is regarded as a cotangent vector in the left-hand subscript and a smooth (linear) function in the right-hand subscript.

\begin{remark}
Let us explain why we have given two names ($\mu$ and $\Mo$) to the same concept.  The Morse groups have the following property which is remarkable in general but trivial in the polyhedral setting: if $f$ is ``generic with respect to $F$'' in some sense (a condition that's irrelevant for polyhedral sheaves), then $\Mo_{x,-f}$ depends only on the first derivatives of $f$ at $x$, up to a shift depending on the second derivatives of $f$ at $x$.  (For instance, if $F^\bullet$ is the constant sheaf, $f$ is any Morse function, and $df$ vanishes at $x$, then the Morse group will be of rank one placed in cohomological degree equal to the index of the Hessian of $f$ at $x$).  In general $\mu_{x,\xi}$ is defined by first forming a sheaf $\nu_x F^\bullet$ called the ``specialization to the tangent space'' of $F^\bullet$, and then taking the Morse group of $\nu_x F^\bullet$ with respect to the function $\xi$.
\end{remark}

\subsection{Singular support and the CCC}

\begin{definition}
Let $V$ be a real vector space and let $F$ be a polyedral sheaf on $V$.  The \emph{singular support} of $F$ is the subset $\SS(F) \subset T^*M$ obtained by taking the closure of all pairs $(x,\xi)$ such that $\mu_{x,\xi} F \neq 0$.
\end{definition}

Singular support can be used to describe constructibility conditions.  For instance, a sheaf $F$ is constructible with respect to a polyhedral stratification $S$ if and only if $\SS(F)$ belongs to the \emph{conormal variety} $\Lambda_S$ of $S$:
$$\Lambda_S:= \bigcup_{s \in S} T^*_s V$$
We have a similar characterization of constructible sheaves which appear in the image of $\kappa$.  Recall we have defined in equation \ref{eq:2} a conical Lagrangian subset of $T^* M_\bR = M_\bR \times N_\bR$ by
$$
\LS = \bigcup_{\tau\in\Sigma} (\tau^\perp + M)\times -\tau$$

\begin{theorem}[\cite{fltz}]
\label{thm:fltz}
Let $\Sigma \subset N_\bR$ be a complete fan and $X$ the corresponding toric variety.  Then $\kappa$ restricts to an equivalence
$$\Perf_T(X) \cong D^b_{cc}(M_\bR;\LS)$$
In particular, if $F^\bullet$ is a sheaf on $M_\bR$ with compact support, and $\SS(F) \subset \LS$, then $F$ is quasi-isomorphic to a $\Theta$-complex, and the corresponding $\Theta'$-complex is perfect.
 \end{theorem}
 
The theorem is an analog of Morelli's theorem \ref{thm:morelli2}, or the rephrasing (4) given in Corollary \ref{cor:morelli3}.  

For fixed $x \in M_\bR$, the microlocal stalks $\mu_{x,\xi}$ assemble to a sheaf on $T_x^* M_\bR$.  More precisely, there is a complex $\mu_{x,\xi}F^\bullet$ of constructible, conical sheaves on $T_x^* M_\bR$ whose stalk at $\xi$ is naturally idenitified with $\mu_{x,\xi}F^\bullet$.  A basic consequence of Theorem \ref{thm:fltz} is that $\mu_x \kappa(\cE^\bullet)$ is constructible with respect to the fan $-\Sigma$.  In particular, we may describe this sheaf ``combinatorially'' as in Section \ref{sec:polyh}.  Let us write $\mu_{x,-\sigma}(F^\bullet)$ for the stalk of $\mu_x$ at any point of the interior of $-\sigma$.

\begin{remark}
We do not define the Fourier-Sato transform for sheaves here, but its existence implies that if $F^\bullet$ is polyhedral, then the sheaf $\mu_x F^\bullet \in D^b_c(N_\bR)$ determines $F^\bullet$ in a neighborhood of $x$.
\end{remark}

The next proposition computes these microlocal stalks for $\Theta$-sheaves.  As a consequence we describe what the coherent counterpart of the microlocal stalks are.

\begin{proposition}
\label{prop:mutheta}
Let $x \in M_\bR$, $\sigma,\tau \in \Sigma$, and $\chi \in M_\bR/\tau^\perp$.  After choosing an orientation of $M_\bR$, we have canonically
$$\mu_{x,-\sigma}(\Theta(\tau,\chi)[\dim(M_\bR)]) = \Bigg\{
\begin{array}{ll}
\sfR & \text{if $\sigma \subset \tau$ and $x$ is in the closed face of $\chi + \tau^\vee$}\\
& \text{corresponding to $\sigma$}\\
0 & \text{otherwise}
\end{array}
$$
\end{proposition}

\begin{proof}
If $x$ is not in the support of $\Theta(\tau,\chi)$ then there is nothing to prove, and so suppose $x \in \chi + \tau^\vee$.  Without loss of generality we may assume $\chi = 0$, $x \in \tau^\vee$.  Let $\upsilon$ be the largest face of $\tau$ that is perpendicular to $x$.  When $\upsilon = 0$, then $x$ is in the interior of $\tau^\vee$ and the proposition is easy to verify.  Suppose that $\upsilon \neq 0$.

The canonical map $\Theta(\tau,0) \hookrightarrow \Theta(\upsilon,0)$ is an isomorphism in a neighborhood of $x$, so we have naturally
$\mu_{x,-\sigma}(\Theta(\tau,0)) \cong \mu_{x,-\sigma}(\Theta(\upsilon,0))$.  The cohomology of the latter 
complex is naturally identified with the relative compactly supported cohomology groups $H^i_c((\upsilon^\vee)^\circ,\{\xi > 1\} \cap (\upsilon^\vee))$, which by excision is isomorphic to $H^i_c((\upsilon^\vee)^\circ \cap \{\xi < 1\})$.  But $(\upsilon^\vee)^\circ \cap \{\xi < 1\}$ is an oriented convex open set when $\xi \in \upsilon$, and empty otherwise.  This completes the proof.
\end{proof}

\begin{theorem}
\label{thm:mu}
Suppose that $F^\bullet$ is a  constructible sheaf on $M_\bR$  with compact support and with $\SS(F) \subset \LS$.  Let $\cE^\bullet$ be the corresponding $T$-equivariant perfect complex on $X$.  Let $\sigma \in \Sigma$ be a cone and let $X_\sigma \subset X$ the corresponding toric subvariety.  If $x \in M_\bZ \subset M_\bR$ is an integral point, and we choose an orientation of $M_\bR$ then we have a natural quasi-isomorphism
$$\mu_{x,-\sigma} (F^\bullet) \cong \RGa(X_\sigma,\cE^\bullet\vert_{X_\sigma})_x$$
where $(-)_x$ denotes the $x$th weight space.
\end{theorem}

\begin{proof}
We will prove the theorem whenever $F^\bullet$ is a $\Theta$-complex and $\cE^\bullet$ is the corresponding $\Theta'$-complex.  That this includes the case where $F^\bullet$ has compact support and $\cE^\bullet$ is perfect is Theorem \ref{thm:fltz}.  Clearly we have natural isomorphisms
$$\RGa(X_\sigma,\Theta'(\tau,\chi)) = \Bigg\{
\begin{array}{ll}
\sfR & \text{if $\sigma \subset \tau$ and $x$ is in the closed face of $\chi + \tau^\vee$}\\
& \text{corresponding to $\sigma$}\\
0 & \text{otherwise}
\end{array}
$$
so to construct the natural isomorphism
$\mu_{x,-\sigma}(F^\bullet) \cong \RGa(X_\sigma,\cE^\bullet\vert_{X_\sigma})_x$ we simply apply propositions \ref{prop:mutheta} and \ref{prop:nat}.
\end{proof}

\subsection{Some examples of $\mu_x$}
\label{sec:seomx}

The most basic computation is the following: for $x \in \chi$, and after choosing an orientation of $M_\bR$, we have
\begin{equation}
\label{eq:3}
\mu_x(\Theta(\sigma,\chi))[\dim(M_\bR)] \cong \sfR_{-\sigma}
\end{equation} 
where the right-hand side denotes the constant sheaf on $-\sigma$, placed in homological degree zero and extended by zero to all of $N_\bR$.  This may be used to deduce all of the following examples:

\begin{example}
If $F$ is constant in a neighborhood of $x$, then $\mu_x(F)$ is a skyscraper sheaf at $0 \in T_x^* M_\bR$.  More precisely, $\mu_{x,0}(F) = F_x[-\dim(M_\bR)]$.
\end{example}

\begin{example}
\begin{enumerate}
 \item
If $F$ is constant on the open interval $(0,1)$, extended by zero to all of $\bR$, then $\mu_0 F$ is the constant sheaf on the nonpositive numbers it $T_0^* \bR \cong \bR$ and $\mu_1 F$ is the constant sheaf on the nonnegative numbers.  Again $\mu_x$ introduces a shift-by-$(-\dim(M_\bR))$.
\item 
If $F$ is constant on the closed interval $[0,1]$, then $\mu_0 F$ is constant on nonnegative numbers and $\mu_1 F$ is constant on nonpositive numbers.  This time $\mu_x$ introduces no shift.
\end{enumerate}
\end{example}

\begin{example}
\label{ex:weirdlocal}
In this example we consider a sheaf $F$ on $\bR^2$ supported on the union of the second and fourth quadrants.  The fourth sheaf of Example \ref{ex:hirz}, and the sheaf of Example \ref{ex:fujino}, are isomorphic in neighborhoods of their ``interesting'' points to this sheaf, at least after a linear change of coordinates.

Let $A$ denote the constant sheaf on the closed second quadrant, and let $B$ denote the constant sheaf on the interior of the fourth quadrant, extended by zero.  One may compute $\Ext^2(A,B) \cong \sfR$.  We may regard a generator $f \in \Ext^2(A,B)$ as a homomorphism $A \to B[2]$.  We let $F$ denote the cone on this map---it is isomorphic to $B[2]$ in the fourth quadrant, and to $A[1]$ in the second quadrant.

We may compute $\mu_0 F$ from the exact triangle
$$\mu_0 A \to \mu_0 B[2] \to \mu_0 F \to$$
From Equation \ref{eq:2} we have $\mu_0 B[2] \cong A$.  A similar computation shows that $\mu_0(A)$ is the constant sheaf on the interior of the second quadrant, extended by zero, and the map $\mu_0 A \to A$ is the canonical inclusion.  Thus the cone $\mu_0 F$ is the constant sheaf supported on $\{(x,0) \mid x \leq 0\} \cup \{(0,y) \mid y \geq 0\}$.
\end{example}

\subsection{Proof of Theorems \ref{thm:vecbund} and  \ref{thm:nefness}}
\label{sec:positivity}

In this section we will apply Theorem \ref{thm:mu} to give a second characterization of those sheaves coming from vector bundles (proving Theorem \ref{thm:vecbund}) and from nef vector bundles (proving Theorem \ref{thm:nefness}).  The second result uses a characterization of nefness for equivariant vector bundles due to Hering, Mustata, and Payne.


\begin{proof}[Proof of Theorem \ref{thm:vecbund}]
A perfect complex $\cE^\bullet$ is concentrated in degree zero if and only if its pulback to a resolution of singularities is concentrated in degree zero.  In \cite[Example 3.11]{fltz} it is shown that pulling back to a toric resolution of singularities has no effect on the associated constructible sheaf, so to prove the Theorem we may assume that $X$ is smooth.

Let $\sigma$ be a top-dimensional cone, and let $X_\sigma \in X$ be the associated $T$-fixed point.  By Theorem \ref{thm:mu}, the fiber of $\cE^\bullet$ at $X_\sigma$ is the direct sum over lattice points $x$ of microlocal stalks $\mu_{x,-\sigma}(\kappa(\cE^\bullet))$---in particular, condition (2) of Theorem \ref{thm:vecbund} holds if and only if the fiber of $\cE^\bullet$ at each $T$-fixed point is concentrated in degree zero.

To complete the proof we have to show that this latter condition is equivalent to $\cE^\bullet$ being concentrated in degree zero.  Since $X$ is smooth we may assume $X = \bA^n$, and since $\cE^\bullet$ is $T$-equivariant it is $\Gm$-equivariant for the standard $\Gm$-action on $\bA^n$.  By Koszul duality \cite{bgg} such a complex is  quasi-isomorphic to a vector bundle if and only if its fiber at $0$ is concentrated in degree zero.
\end{proof}

\begin{proof}[Proof of Theorem \ref{thm:nefness}]
For $\sigma$ a top-dimensional cone let $X_\sigma \in X$ be the associated fixed point, and for $\tau$ a codimension-one cone let $X_\tau$ denote the associated invariant curve. 
By the theorem of \cite{HMP}, $\cE$ is nef if and only if its restriction to each $X_\tau$ is nef.  We claim that this is the case if and only if $H^i(\cE\vert_{X_{\tau}})$ vanishes for $i < 0$ and $H^0(\cE\vert_{X_\tau}) \to \cE_{x_\sigma}$ is a surjection for every $\sigma \supset \tau$.  Indeed if the higher cohomology of $\cE\vert_{X_\tau}$ vanishes, then it cannot have a summand of the form $\cO(k)$ for $k < -1$, and if the surjectivity condition holds then it cannot have a summand of the form $\cO(-1)$.

A map of $T$-modules is surjective if and only if the induced map on $x$-weight spaces is surjective for every character $x:T \to \Gm$.  Thus the condition that $H^0(\cE\vert_{X_\tau}) \to \cE_{X_\sigma}$ is surjective for every $\sigma \supset \tau$ is equivalent to the condition that $H^0(\cE\vert_{X_\tau})_\chi \to \cE_{X_\sigma,x}$ is surjective for every $x$.  Theorem \ref{thm:mu} completes the proof. 
\end{proof}

\emph{Acknowledgments:}  I thank Bohan Fang, Melissa Liu, Sam Payne, David Speyer and Eric Zaslow for useful comments and corrections.

\end{document}